\documentclass[11pt,letterpaper]{amsart}
\usepackage[T1]{fontenc}
\usepackage{amsmath,amssymb,amsfonts,mathtools}
\usepackage{microtype}
\usepackage{enumitem}
\usepackage{booktabs,array}
\newcolumntype{P}[1]{>{\raggedright\arraybackslash}p{#1}}
\usepackage[hidelinks]{hyperref}

\numberwithin{equation}{section}
\allowdisplaybreaks
\newtheorem{theorem}{Theorem}[section]
\newtheorem{maintheorem}[theorem]{Main theorem}
\newtheorem{lemma}[theorem]{Lemma}
\newtheorem{proposition}[theorem]{Proposition}
\newtheorem{corollary}[theorem]{Corollary}
\newtheorem{problem}[theorem]{Problem}

\theoremstyle{definition}
\newtheorem{definition}[theorem]{Definition}
\newtheorem{setup}[theorem]{Setup}

\DeclareMathOperator{\Stab}{Stab}
\newcommand{\Z}{\mathbb Z}
\newcommand{\F}{\mathcal F}
\newcommand{\G}{\mathcal G}
\newcommand{\Hh}{\mathcal H}
\newcommand{\1}{\mathbf 1}

\title[One extra edge forces Berge pancyclicity]{One extra edge forces Berge pancyclicity}
\author{Henry Shin}
\date{}


\hypersetup{
  pdftitle={One extra edge forces Berge pancyclicity},
  pdfauthor={Henry Shin},
  pdfsubject={Resolution of the Bailey--Hollars--Li--Luo one-extra-edge question},
  pdfkeywords={Berge cycle, pancyclic hypergraph, Hamiltonian hypergraph, matching exchange, sum-free stability}
}

\begin{document}

\begin{abstract}
We resolve a question of Bailey, Hollars, Li and Luo.  For sufficiently
large $n$, set $r=\lfloor(n-1)/2\rfloor$.  We prove that the edges of any
Hamiltonian Berge cycle in a simple $n$-vertex $r$-uniform hypergraph,
together with any one additional edge, contain Berge cycles of every length
from $2$ to $n$.  In odd order we prove a stronger prescribed-unused-edge
theorem using rigidity of large subsets of odd cyclic groups and an
alternating matching exchange.  In even order we introduce a two-gap
edge-reassignment method.  Counting, additive structure, and split-lock
arguments cover all lengths outside a seven-term middle band.  The absence
of the central length forces an exact reflected translation-wave structure,
which is eliminated by an additive covering theorem derived from
sum-free stability.  The remaining
near-central lengths follow from a two-defect recurrence and bounded-run
forcing.
\end{abstract}

\maketitle

\section{Introduction}

All hypergraphs in this paper are finite and simple, meaning that their
edge families consist of distinct subsets of the vertex set; repeated
hyperedges are not allowed.  A \emph{Berge cycle} of length $\ell\ge2$
in a hypergraph $\Hh$ is an alternating sequence
\[
 x_0h_0x_1h_1\cdots x_{\ell-1}h_{\ell-1}x_0
\]
in which the vertices $x_0,\ldots,x_{\ell-1}$ are distinct, the hyperedges
$h_0,\ldots,h_{\ell-1}$ are distinct, and
$\{x_i,x_{i+1}\}\subseteq h_i$ for every $i$, with indices modulo $\ell$.
For such a cycle $C$, write
\[
        E(C):=\{h_0,\ldots,h_{\ell-1}\}.
\]
A Berge cycle is \emph{Hamiltonian} if it contains every vertex.  An
$n$-vertex hypergraph is \emph{Berge-pancyclic} if it contains Berge cycles
of every length $2,3,\ldots,n$.  Thus length $2$ is included: a Berge
$2$-cycle consists of two distinct vertices and two distinct hyperedges
containing both of them.  This is the convention of Bailey--Hollars--Li--Luo
\cite{BHLL}; some earlier work uses lengths $3,\ldots,n$ instead
\cite{BLL}.  For a positive integer $n$, we write
$\Z_n:=\mathbb Z/n\mathbb Z$.

Berge cycles were introduced by Berge~\cite{Berge}, and early
Hamiltonicity results for this notion go back to Bermond, Germa,
Heydemann and Sotteau~\cite{BermondEtAl}.  Dirac's Hamiltonicity theorem
and Bondy's pancyclic strengthening provide the classical graph-theoretic
paradigm~\cite{Dirac,Bondy}.  For Berge cycles, fixed-length and
consecutive-length extremal questions were studied by Gy\H{o}ri and
Lemons~\cite{GyoriLemons} and by Jiang and Ma~\cite{JiangMa}.

Sharp results for long cycles include work of F\"uredi, Kostochka and
Luo~\cite{FKL-long}; Dirac-type Hamiltonicity was developed in the
non-uniform setting by the same authors~\cite{FKL-nonuniform} and in the
uniform setting by Kostochka, Luo and McCourt~\cite{KLM}.
Super-pancyclicity was studied by Kostochka, Luo and
Zirlin~\cite{KLZ}, while Salia obtained P\'osa-type degree
conditions~\cite{Salia}.  Bailey,
Li and Luo, and then Bailey, Hollars, Li and Luo, obtained sharp
Dirac-type pancyclicity results across broad ranges of uniformities
~\cite{BLL,BHLL}.  At the remaining half-uniformity borderline, Bailey,
Hollars, Li and Luo asked the following question~\cite[Question~1]{BHLL}.

\begin{problem}[Bailey--Hollars--Li--Luo]\label{prob:BHLL}
Let $n$ be sufficiently large and put
\[
        r=\left\lfloor\frac{n-1}{2}\right\rfloor.
\]
Suppose that an $n$-vertex $r$-uniform hypergraph contains a Hamiltonian
Berge cycle $C$ and at least one edge outside $E(C)$.  Must the hypergraph
be Berge-pancyclic?
\end{problem}

Our main theorem answers this question affirmatively.

\begin{maintheorem}[Resolution of the one-extra-edge question]\label{thm:main}
There exists $n_0$ such that the following holds for every $n\ge n_0$.
Put
\[
        r=\left\lfloor\frac{n-1}{2}\right\rfloor.
\]
If a simple $n$-vertex $r$-uniform hypergraph $\Hh$ contains a Hamiltonian
Berge cycle $C$ and an edge $f\notin E(C)$, then the subhypergraph with edge
set $E(C)\cup\{f\}$ is Berge-pancyclic.
\end{maintheorem}

The conclusion is local to the chosen Hamiltonian cycle and the chosen
extra edge: the $n+1$ edges of $E(C)\cup\{f\}$ already contain Berge
cycles of every length $2,\ldots,n$.

The two parities require different mechanisms.  Odd order admits the
following stronger theorem.  Relative to a displayed Hamiltonian cyclic
order, a hyperedge $h$ \emph{realizes cyclic distance} $d$ if it contains a
pair $\{v_i,v_{i+d}\}$ for some $i$.

\begin{theorem}[Prescribed leftover chords in odd order]\label{thm:odd-prescribed}
Let $n=2r+1\ge5$.  Let $\Hh$ be a simple $n$-vertex hypergraph containing a
Hamiltonian Berge cycle
\[
        C=v_0e_0v_1e_1\cdots v_{n-1}e_{n-1}v_0.
\]
Let $\G\subseteq E(\Hh)$ satisfy $E(C)\subseteq\G$, and suppose that every
hyperedge in $\G$ has size at least $r$.  If
$D\subseteq\{2,\ldots,r\}$ and $|\G|\ge n+|D|$, then there is a Hamiltonian
Berge cycle $C_D$ with the same vertex sequence and edge set contained in
$\G$, together with distinct hyperedges
$g_d\in\G\setminus E(C_D)$, one for each $d\in D$, such that $g_d$ realizes
cyclic distance $d$.
\end{theorem}

\begin{corollary}[Odd one-extra-edge theorem]\label{cor:odd-one-extra}
Let $n=2r+1\ge7$.  If a simple $n$-vertex hypergraph contains a Hamiltonian
Berge cycle $C$ whose hyperedges have size at least $r$ and an additional
hyperedge $f\notin E(C)$ of size at least $r$, then the hyperedges in
$E(C)\cup\{f\}$ form a Berge-pancyclic hypergraph.  In particular, the odd-order case of
Problem~\ref{prob:BHLL} holds in every odd order $n\ge7$.
\end{corollary}

For even order we prove the following asymptotic statement.

\begin{theorem}[Even one-extra-edge theorem]\label{thm:main-even}
There exists $m_0$ such that the following holds for every $m\ge m_0$.
Let $\Hh$ be a simple $(m-1)$-uniform hypergraph on $2m$ vertices, let $C$
be a Hamiltonian Berge cycle in $\Hh$, and let
$f\in E(\Hh)\setminus E(C)$.  Then the subhypergraph with edge set
$E(C)\cup\{f\}$ is Berge-pancyclic.
\end{theorem}

The parity obstruction is genuine.  In $\Z_{2m}$ the $(m-1)$-set
\[
        \{0,2,4,\ldots,2m-4\}
\]
misses every odd cyclic distance.  Thus the odd-order strategy of leaving a
single chord edge unused cannot extend directly.  The replacement is a
\emph{free-label two-gap construction}: two intervals of Hamiltonian labels
are freed simultaneously, and labels from their union are reassigned to
the two bridge pairs.  A label freed by one gap may therefore cover the
bridge across the other.

Table~\ref{tab:proof-map} summarizes the proof architecture.

\begin{table}[ht]
\centering
\small
\renewcommand{\arraystretch}{1.14}
\begin{tabular}{@{}P{0.31\linewidth}P{0.61\linewidth}@{}}
\toprule
Case & Main mechanism \\
\midrule
Odd order (Section~\ref{sec:odd-order})
& Cyclic-distance rigidity and alternating matching exchange. \\
Even endpoints and outside band (Section~\ref{sec:outside})
& Counting, Kneser's theorem, split locks, and an antipodal quotient. \\
Even central length $m+1$ (Sections~\ref{sec:central-length} and~\ref{sec:cover-surplus})
& Exact reflected-wave extraction and a cover surplus from sum-free stability. \\
Other six even middle lengths (Section~\ref{sec:six-middle})
& A two-defect recurrence and all-split forcing. \\
\bottomrule
\end{tabular}
\caption{The four components of the proof of Theorem~\ref{thm:main}.}
\label{tab:proof-map}
\end{table}

The central length $m+1$ is the main structural point in even order.  If no
such cycle exists, the additional edge and every Hamiltonian edge are
antipodal-free.  Encoding each
row by signs on antipodal pairs gives a spin array with exactly one zero in
every row.  Long-range two-gap triggers force every noncentral column to be
monochromatic, with opposite signs in reflected columns.  Hence the system
is an exact reflected translation wave with at most one deleted base offset
per row.

It remains to show that some two-gap template survives all these deletions.
After folding by the reflection symmetry, the number of available bridge
assignments becomes an additive correlation $P_N(J)$ on a near-transversal
of an interval.  We prove the uniform quadratic bound
\[
        P_N(J)\ge cH^2
\]
for every sufficiently large $H$, using the sum-free stability consequence
recorded by Balogh, Liu, Sharifzadeh and Treglown~\cite{BLST}.  After
puncturing one folded representative, the resulting quadratic surplus
dominates the maximum one-row deletion degree.  Rotating the templates then
absorbs the at-most-one deleted offset in each row and leaves a surviving
two-gap cycle.

For the other six middle lengths, if the additional edge misses the
near-central distance $s=m-j$, then its indicator word satisfies
\[
        \alpha_{x+2j}-\alpha_x=\delta_x-\delta_{x-s},
        \qquad j\in\{1,2,3\},
\]
where $\delta$ is supported on two defect positions.  For $j=2,3$ this
forces uniformly short monochromatic runs; two adjacent additive diagonals
then make all but $O(1)$ vertices forbidden in a defect row.  The case
$j=1$ has an explicit dihedral normal form, with its single interval orbit
eliminated by simplicity and a Hamiltonian-edge exchange.

The argument is deterministic.  Its only asymptotic ingredient is the
sum-free stability consequence just described.

\section{Cyclic preliminaries}\label{sec:prelim}

All cyclic indices are taken modulo the ambient order.  For a positive
integer $q$, write $[q]=\{1,\ldots,q\}$.  For ordinary integers $u,v$, set
\[
 [u,v]=\{u,u+1,\ldots,v\}\quad\text{when }u\le v,
 \qquad [u,v]=\varnothing\quad\text{when }u>v.
\]
If $u\equiv v\pmod 2$, also set
\[
 [u,v]_2=\{u,u+2,\ldots,v\}\quad\text{when }u\le v,
 \qquad [u,v]_2=\varnothing\quad\text{when }u>v.
\]
For subsets $X,Y$ of an abelian group $G$ and $t\in G$, write
\[
 X+Y=\{x+y:x\in X,\ y\in Y\},\qquad
 X-Y=\{x-y:x\in X,\ y\in Y\},
\]
\[
 X+t=\{x+t:x\in X\},\qquad X-t:=X+(-t),\qquad -X=\{-x:x\in X\}.
\]
For scalar-left operations we use the singleton convention
\[
        t+X:=X+t,
        \qquad
        t-X:=t+(-X)=\{t-x:x\in X\}.
\]
When $X\subseteq G$, the notation $X^c$ always means the ambient
complement $G\setminus X$.  For a statement $\mathcal P$, let
$\1_{\mathcal P}$ be $1$ when $\mathcal P$ is true and $0$ otherwise;
for a set $X$, we also write $\1_X(x)=\1_{x\in X}$.
For $a,b\in\Z_n$, write $[a,b]^+$ for the clockwise interval from $a$ to
$b$, including both endpoints.  In $\Z_{2m}$ an \emph{antipodal pair} is
a pair $\{x,x+m\}$.  If $S$ is a subset of a finite abelian
group $G$, its stabilizer is
\[
        \Stab(S):=\{g\in G:S+g=S\}.
\]
For a graph $\Gamma$ and a set $X\subseteq V(\Gamma)$, write
$N_\Gamma(X)$ for the open neighborhood of $X$.  For sets $X,Y$, their
symmetric difference is
\[
        X\triangle Y=(X\setminus Y)\cup(Y\setminus X).
\]
For a function or finite sequence $g$, put
\[
        \operatorname{supp}(g)=\{x:g(x)\ne0\}.
\]
The symbol $\oplus$ denotes addition modulo $2$.  On the sign alphabet
$\{+,-\}$, the notation $-\sigma$ denotes the opposite sign.  If $\xi$ is
a symbol and $t\ge0$, then $\xi^t$ denotes the word consisting of $t$
consecutive copies of $\xi$.  We use the convention $\max\varnothing=0$.
For $s\in\Z_n\setminus\{0\}$, the \emph{step-$s$ graph} $\Gamma_s(n)$ is the graph on
$\Z_n$ with edge set
\[
        \bigl\{\{x,x+s\}:x\in\Z_n\bigr\};
\]
its connected components are the orbits of translation by $s$.
For $A\subseteq\Z_n$, define
\[
        \Delta(A)=\{d\in\{1,\ldots,\lfloor n/2\rfloor\}:A\cap(A+d)\ne\varnothing\}.
\]
Thus $d\in\Delta(A)$ exactly when $A$ contains a pair at cyclic distance
$d$.  We write $(x)_+=\max\{x,0\}$ for the positive part of a real number
$x$.

\begin{lemma}[Chord lemma]\label{lem:chord}
Let $\Hh$ contain a Hamiltonian Berge cycle
\[
        C=v_0e_0v_1e_1\cdots v_{n-1}e_{n-1}v_0.
\]
Let $h\in E(\Hh)\setminus E(C)$.  If
$\{v_i,v_{i+s}\}\subseteq h$ for some $1\le s\le n-1$, then $\Hh$
contains Berge cycles of lengths
\[
        s+1\qquad\text{and}\qquad n-s+1.
\]
\end{lemma}

\begin{proof}
The two Hamiltonian arcs between $v_i$ and $v_{i+s}$ have lengths $s$ and
$n-s$.  Adding the unused hyperedge $h$ to either arc gives a Berge cycle.
If one arc has length $1$, this is a Berge $2$-cycle; the two edges are
distinct because $h\notin E(C)$.
\end{proof}

\begin{lemma}[Odd cyclic distance rigidity]\label{lem:odd-distance}
Let $n=2r+1$, let $d\in\{1,\ldots,r\}$, and let
$A\subseteq\Z_n$ have size at least $r$.  If $d\notin\Delta(A)$, then
$|A|=r$, $\gcd(n,d)=1$, and $A$ is a maximum independent set in the cycle
on $\Z_n$ with edge set
\[
        \{\{x,x+d\}:x\in\Z_n\}.
\]
In particular, $A$ misses no cyclic distance other than $d$.  Moreover,
for fixed $d$ there are at most $n$ subsets of $\Z_n$ of size at least $r$
that miss $d$.  When $\gcd(n,d)=1$, these are exactly the $n$ maximum
independent sets of the $d$-step cycle.
\end{lemma}

\begin{proof}
If $d\notin\Delta(A)$, then $A\cap(A+d)=\varnothing$.  Let $G_d$ be the
graph on $\Z_n$ with edges $\{x,x+d\}$.  If $g=\gcd(n,d)$, then $G_d$ is
the disjoint union of $g$ odd cycles of length $n/g$, and its independence
number is
\[
        g\frac{n/g-1}{2}=\frac{n-g}{2}.
\]
Since $|A|\ge(n-1)/2$, we must have $g=1$ and $|A|=r$.  Thus $A$ is a
maximum independent set in the single odd cycle $G_d$.

Writing the vertices of $G_d$ as
\[
        0,d,2d,\ldots,(n-1)d,
\]
the gaps between consecutive selected vertices of $A$ contain at least one
unselected vertex.  There are $r$ selected and $r+1$ unselected vertices,
so exactly one gap contains two unselected vertices and all other gaps
contain one.  After translation,
\[
        A=\{0,2d,4d,\ldots,2(r-1)d\}.
\]
Multiplying by $d^{-1}$, it is enough to inspect
\[
        E=\{0,2,4,\ldots,2(r-1)\}\subseteq\Z_{2r+1}.
\]
The nonzero differences of $E$ are the residues $2t$ with
$-(r-1)\le t\le r-1$, $t\ne0$.  These are precisely all nonzero residues
except $\pm1$.  Multiplying back, the only missing nonzero residue pair is
$\{\pm d\}$.  Hence $A$ misses no cyclic distance other than $d$.

For fixed $d$, if $\gcd(n,d)>1$ no such $A$ exists.  If
$\gcd(n,d)=1$, the sets are precisely the $n$ maximum independent sets of
the $d$-step odd cycle, obtained by placing the unique two-zero gap in any
position.
\end{proof}

\section{Odd order: prescribed leftover chords}\label{sec:odd-order}

Let $n=2r+1$.  Let $P=\Z_n$ be the set of adjacent-pair positions in the
fixed cyclic order, where position $i$ corresponds to the pair
$\{i,i+1\}$.  Given a finite family $\F$ of subsets of $\Z_n$, define the
bipartite graph $B(\F)$ with parts $\F$ and $P$ by joining
$F\in\F$ to $i\in P$ when $\{i,i+1\}\subseteq F$.

\begin{proposition}[One-distance exchange]\label{prop:one-distance}
Let $n=2r+1\ge5$, let $2\le d\le r$, and let $\F$ be a finite family of
distinct subsets of $\Z_n$, each of size at least $r$, with $|\F|>n$.
Suppose that $B(\F)$ has a matching saturating $P$.  Then $B(\F)$ has a
matching saturating $P$ for which at least one unmatched member of $\F$
realizes distance $d$.
\end{proposition}

\begin{proof}
Choose a matching $M$ saturating $P$.  If an unmatched member realizes
$d$, there is nothing to prove.  Otherwise choose an unmatched member
$F_*$.  Then $F_*$ misses $d$.  By Lemma~\ref{lem:odd-distance}, at most
$n$ members of $\F$ miss $d$; since $|\F|>n$, some member realizes $d$.

Orient the edges of $B(\F)$ alternately with respect to $M$: unmatched
edges from $\F$ to $P$, and matched edges from $P$ to $\F$.  Let
$T\subseteq\F$ and $S\subseteq P$ be the vertices reachable from $F_*$.  If
$T$ contains a member $G$ that realizes $d$, then switching along an
alternating path from $F_*$ to $G$ leaves $G$ unmatched while still
saturating $P$.  We may therefore assume that every member of $T$ misses
$d$.

We record the alternating-reachability relation explicitly.  Every
neighbor of a reachable family vertex is reachable: an unmatched incident
edge is directed from $\F$ to $P$, while the unique matched edge incident
with a reachable family vertex other than $F_*$ is the edge by which that
vertex was reached.  Conversely, every reachable position is reached from
a member of $T$.  Hence
\[
        S=N_{B(\F)}(T).
\]
Because $M$ saturates $P$, every $s\in S$ has a matched partner in $\F$;
the matched edge is directed from $s$ to that partner, which therefore lies
in $T$.  Conversely, every member of $T\setminus\{F_*\}$ is reached through
its unique matched edge from a position in $S$.  No other unmatched family
vertex can be reachable, since all directed edges entering the family side
are matched edges.  Thus $M$ pairs $S$ bijectively with
$T\setminus\{F_*\}$, and
\begin{equation}\label{eq:alternating-reachability}
        S=N_{B(\F)}(T),\qquad |S|=|T|-1.
\end{equation}
Choose $A\in T$, and let
\[
        Q=\{i:\{i,i+1\}\subseteq A\}.
\]
The set $Q$ is nonempty; otherwise $A$ would miss both distances $1$ and
$d$, contrary to Lemma~\ref{lem:odd-distance} because $d>1$.

By Lemma~\ref{lem:odd-distance}, all members of $T$ are distinct translates
of $A$.  Thus $T=\{A+i:i\in I\}$ for some $I\subseteq\Z_n$, and
\[
        N_{B(\F)}(T)=I+Q.
\]
For any fixed $q\in Q$, the translate $I+q$ has size $|I|=|T|$ and is
contained in $I+Q$, contradicting
\eqref{eq:alternating-reachability}.
\end{proof}

\begin{theorem}[Simultaneous leftover distances]\label{thm:odd-setsystem}
Let $n=2r+1\ge5$, and let $\F$ be a finite family of distinct subsets of
$\Z_n$, each of size at least $r$.  Suppose that $B(\F)$ has a matching
saturating $P$.  If $D\subseteq\{2,\ldots,r\}$ and
\[
        |\F|\ge n+|D|,
\]
then there is a matching in $B(\F)$ saturating $P$, together with distinct
unmatched members $F_d\in\F$, one for each $d\in D$, such that $F_d$
realizes distance $d$.
\end{theorem}

\begin{proof}
Induct on $|D|$.  The case $D=\varnothing$ is immediate.  Choose $d\in D$.
Proposition~\ref{prop:one-distance} gives a matching saturating $P$ and
leaving a $d$-realizing member $F_d$ unmatched.  Remove $F_d$ and apply the
induction hypothesis to $D\setminus\{d\}$; the existing matching still
saturates $P$ in the smaller family.
\end{proof}

\begin{proof}[Proof of Theorem~\ref{thm:odd-prescribed}]
Identify $v_i$ with $i\in\Z_n$, and let
\[
        \F=\bigl\{\{i:v_i\in h\}:h\in\G\bigr\}.
\]
Simplicity makes the map from $\G$ to $\F$ injective.  The displayed
Hamiltonian cycle gives a matching in $B(\F)$ saturating $P$, by matching
$e_i$ to position $i$.  Theorem~\ref{thm:odd-setsystem} gives the required
reassignment and unmatched sets.  Translating back to hyperedges proves the
theorem.
\end{proof}

\begin{lemma}[Length two in odd order]\label{lem:two-cycle-odd}
Let $n=2r+1\ge7$.  A family of at least $n+1$ distinct subsets of an
$n$-element set, each of size at least $r$, contains two members sharing a
pair of points.  Consequently the corresponding hypergraph contains a
Berge cycle of length $2$.
\end{lemma}

\begin{proof}
If no two members shared a pair, then
\[
        (n+1)\binom r2\le\binom n2.
\]
Substituting $n=2r+1$ gives
\[
        (2r+2)\binom r2\le r(2r+1),
\]
but
\[
        (2r+2)\binom r2-r(2r+1)=r(r^2-2r-2)>0
\]
for $r\ge3$, a contradiction.
\end{proof}

\begin{corollary}[Local odd-order pancyclicity]\label{cor:local-odd}
Let $n=2r+1\ge7$.  Let $\Hh$ be a simple $n$-vertex hypergraph containing
a Hamiltonian Berge cycle $C$.  Let $\G\subseteq E(\Hh)$ satisfy
$E(C)\subseteq\G$ and $|\G|>n$, and suppose that every hyperedge in $\G$ has
size at least $r$.  Then the subhypergraph with edge set $\G$ is
Berge-pancyclic.
\end{corollary}

\begin{proof}
Length $n$ is supplied by $C$, and length $2$ by
Lemma~\ref{lem:two-cycle-odd}.  For $d\in\{2,\ldots,r\}$, apply
Theorem~\ref{thm:odd-prescribed} with $D=\{d\}$.  We obtain a Hamiltonian
cycle $C_d$ and an unused hyperedge realizing distance $d$.
Lemma~\ref{lem:chord} gives lengths $d+1$ and $n-d+1$.  As $d$ ranges from
$2$ to $r$, these are exactly the lengths $3,4,\ldots,n-1$.
\end{proof}

\begin{proof}[Proof of Corollary~\ref{cor:odd-one-extra}]
Apply Corollary~\ref{cor:local-odd} with $\G=E(C)\cup\{f\}$.
\end{proof}

\section{The free-label two-gap certificate}\label{sec:two-gap}

For the remainder of the even-order proof, unless stated otherwise, write
the fixed Hamiltonian Berge cycle as
\[
        0\,e_0\,1\,e_1\cdots(2m-1)e_{2m-1}0
\]
on $\Z_{2m}$.  A two-gap template is specified by integers
$1\le d,e\le2m-1$ and two disjoint cyclic intervals of adjacent-pair
indices
\[
        I=[x,x+d-1]^+,
        \qquad
        J=[y,y+e-1]^+,
\]
where $|I|=d$ and $|J|=e$.  In particular, disjointness implies
$d+e\le2m$.
Deleting the open vertex gaps corresponding to $I$ and $J$ leaves two
retained arcs.  The bridge pairs are
\[
        \beta_I=\{x,x+d\},
        \qquad
        \beta_J=\{y,y+e\}.
\]
The labels indexed by $I\cup J$ are free; all retained adjacent pairs may
use their original labels.

\begin{proposition}[Free-label two-gap certificate]\label{prop:free-label}
Let $1\le d,e\le2m-1$, and let
$I=[x,x+d-1]^+$ and $J=[y,y+e-1]^+$ be disjoint intervals of
adjacent-pair indices in $\Z_{2m}$, so that $|I|=d$ and $|J|=e$.  If there
exist distinct
$i,j\in I\cup J$ such that
\[
        \beta_I\subseteq e_i,
        \qquad
        \beta_J\subseteq e_j,
\]
then the Hamiltonian labels contain a Berge cycle of length
\[
        2m-(d-1)-(e-1)=2m-d-e+2.
\]
\end{proposition}

\begin{proof}
Use $e_i$ and $e_j$ on the two bridges.  If $d+e=2m$, the two disjoint
index intervals partition the Hamiltonian order.  Their bridge pairs are
the same unordered pair, so the distinct labels $e_i,e_j$ form a Berge
$2$-cycle.

If $d+e=2m-1$, there is a unique index $t\notin I\cup J$.  The two bridge
pairs and the retained adjacent pair $\{t,t+1\}$ are the three sides of a
triangle on the three boundary vertices of the two intervals.  Use
$e_i,e_j,e_t$ on these sides.  The labels are distinct because
$i,j\in I\cup J$ and $t\notin I\cup J$.

Assume now that $d+e\le2m-2$.  Rotate the cyclic order and choose
integer lifts so that
\[
        I=[0,d-1],\qquad J=[y,y+e-1],
        \qquad d\le y\le2m-e.
\]
Starting at the vertex $0$, use $e_i$ on the bridge $0d$, follow the
Hamiltonian labels $e_d,e_{d+1},\ldots,e_{y-1}$ from $d$ to $y$, use
$e_j$ on the bridge from $y$ to $y+e$ modulo $2m$, and then follow
$e_{y+e},e_{y+e+1},\ldots,e_{2m-1}$ back to $0$.  Either displayed
Hamiltonian-label interval may be empty: when $y=d$ the first retained
vertex arc degenerates to the common boundary vertex $d$, and when
$y+e=2m$ the second degenerates to the vertex $0$.  Thus the construction
still forms one cyclic component in the adjacent cases.  The bridge labels
are distinct, and all original Hamiltonian labels used on the retained arcs
have indices outside $I\cup J$.  The vertex count is
\[
        1+(y-d+1)+(2m-y-e)=2m-d-e+2,
\]
which proves the claim.
\end{proof}

\section{Lengths outside the middle band}\label{sec:outside}

We use the following setup throughout this section.

\begin{setup}[Even one-extra-edge setup]\label{setup:even-outside}
Let $n=2m$, $m\ge5$, and let
\[
        C=0\,e_0\,1\,e_1\cdots(2m-1)e_{2m-1}0
\]
be a Hamiltonian $(m-1)$-uniform Berge cycle on vertex set $\Z_{2m}$.  Let
\[
        f\notin\{e_0,\ldots,e_{2m-1}\}
\]
be an additional $(m-1)$-edge, and put
\[
        \mathcal G=\{e_0,\ldots,e_{2m-1},f\}.
\]
\end{setup}

The following form of Kneser's theorem will be used in the distance-set
reduction; the stabilizer notation was defined in
Section~\ref{sec:prelim}.

\begin{theorem}[Kneser~\cite{Kneser}]\label{thm:kneser}
Let $G$ be a finite abelian group and let $X,Y\subseteq G$.  If
$H=\Stab(X+Y)$, then
\[
        |X+Y|\ge |X+H|+|Y+H|-|H|.
\]
\end{theorem}

\begin{lemma}[Length two]\label{lem:even-two-cycle}
Let $m\ge5$.  Any set of $2m+1$ distinct $(m-1)$-subsets of a $2m$-element set contains two members sharing a pair of vertices.  Consequently every setup satisfying Setup~\ref{setup:even-outside} contains a Berge cycle of length $2$.
\end{lemma}

\begin{proof}
If no two members shared a pair, then
\[
        (2m+1)\binom{m-1}{2}\le \binom{2m}{2},
\]
which is false for $m\ge5$.
\end{proof}

\begin{lemma}[Berge triangle]\label{lem:even-triangle}
Let $m\ge8$.  Any set of $2m+1$ distinct $(m-1)$-subsets of a $2m$-element set contains a Berge triangle.  Consequently every setup satisfying Setup~\ref{setup:even-outside} with $m\ge8$ contains a Berge cycle of length $3$.
\end{lemma}

\begin{proof}
The total degree is $(2m+1)(m-1)$, so some vertex $v$ has degree at least
$m$.  Choose any $m$ edges containing $v$.  Within this chosen subfamily,
the total number of incidences with the other $2m-1$ vertices is at least
$m(m-2)$.  Hence some $w\ne v$ has codegree
\[
        c\ge \left\lceil\frac{m(m-2)}{2m-1}\right\rceil
\]
within the chosen subfamily.  Let $\mathcal F_{vw}$ be these $c$ chosen
edges containing $\{v,w\}$.  If the sets $E\setminus\{v,w\}$ for $E\in\mathcal F_{vw}$ were pairwise disjoint, then $c(m-3)\le2m-2$.  This is impossible for $m\ge8$, since
\[
        \frac{m(m-2)(m-3)}{2m-1}>2m-2.
\]
Thus two distinct edges $F_1,F_2\in\mathcal F_{vw}$ share a vertex $x\notin\{v,w\}$.  Since $c\ge3$, choose a third edge $C\in\mathcal F_{vw}\setminus\{F_1,F_2\}$.  The three edges $C,F_1,F_2$ realize respectively $\{v,w\}$, $\{v,x\}$, and $\{w,x\}$, giving a Berge triangle.
\end{proof}

\begin{lemma}[Boundary bypass]\label{lem:boundary-bypass}
Let
\[
        C=0\,e_0\,1\,e_1\cdots(n-1)e_{n-1}0
\]
be a Hamiltonian Berge cycle.  Let $1\le d\le n-2$.  If, for some $i$,
\[
        i-d\in e_i\qquad\text{or}\qquad i+d+1\in e_i,
\]
then the Hamiltonian labels contain a Berge cycle of length $n-d$.
\end{lemma}

\begin{proof}
If $i-d\in e_i$, traverse the Hamiltonian arc from $i+1$ to $i-d$ and close with $e_i$.  This arc has length $n-d-1$, so the resulting cycle has length $n-d$.  The other case is symmetric.
\end{proof}

\begin{lemma}[Distance-two adjacent-row lock]\label{lem:s2-adjacent-lock}
Assume Setup~\ref{setup:even-outside}.  Suppose $f\cap(f+2)=\varnothing$ and no Berge cycle of length $2m-1$ exists.  If $p,p+1\in f$, then
\[
        e_p\setminus f\subseteq T_2^+(f)\cap T_2^-(f),
\]
where
\[
        T_2^+(A)=\{x\notin A:x+2\notin A\},\qquad
        T_2^-(A)=\{x\notin A:x-2\notin A\}.
\]
\end{lemma}

\begin{proof}
Place $f$ at the adjacent pair $\{p,p+1\}$, freeing $e_p$.  If $e_p$ realized cyclic distance $2$, the chord lemma applied to the new Hamiltonian cycle would give length $2m-1$.  Hence $e_p$ also misses distance $2$.

We first prove containment in $T_2^+(f)$.  If
$a\in e_p\cap(f-2)$ and $a\notin\{p-1,p\}$, then the sequence
\[
        p+1,p+2,\ldots,a,\ p,p-1,\ldots,a+2,\ p+1
\]
uses $e_p$ for the jump from $a$ to $p$ and $f$ for the jump from
$a+2$ to $p+1$, giving a Berge cycle of length $2m-1$.  Thus
$e_p\cap(f-2)\subseteq\{p-1,p\}$.  The boundary bypass forbids
$p-1\in e_p$, and $p\in f$, so no vertex of $e_p\setminus f$ lies in
$f-2$.

For the reverse containment, rotate so that $p=0$ and suppose that
$a\in e_0\cap(f+2)$ but $a\notin f$.  The residues $0,1$ lie in $f$,
so $a\notin\{0,1\}$.  The case $a=2$ is excluded by the boundary-bypass
lemma with $d=1$, and the case $a=-1$ is excluded by its other direction.
Thus choose the representative $3\le a\le2m-2$.  In the integer lift, the
sequence
\[
 0,-1,-2,\ldots,a-2m,\ 1,2,\ldots,a-2,\ 0
\]
contains every residue except $a-1$.  Use the original Hamiltonian labels
on the two displayed arcs, use $e_0$ on the jump $a\to1$, and use $f$ on
the jump $a-2\to0$.  The two arc-label intervals are disjoint and omit
$e_0$, while $f$ is additional.  Hence all labels are distinct and the
sequence is a Berge cycle of length $2m-1$, a contradiction.  Therefore
no vertex of $e_p\setminus f$ lies in $f+2$, which proves containment in
$T_2^-(f)$.
\end{proof}

\begin{lemma}[Distance-two defect form]\label{lem:s2-defect-form}
Let $m\ge3$.  Let $A\subseteq\Z_{2m}$ have size $m-1$ and satisfy
$A\cap(A+2)=\varnothing$.
\begin{enumerate}[label=\textup{(\roman*)}]
\item If $m$ is odd, then $T_2^+(A)\cap T_2^-(A)=\varnothing$.
\item If $T_2^+(A)\cap T_2^-(A)\ne\varnothing$, then $m$ is even and there is a unique vertex $u\in T_2^+(A)\cap T_2^-(A)$ such that, with $F=A\cup\{u\}$,
\[
        F+2=F^c.
\]
\end{enumerate}
\end{lemma}

\begin{proof}
The step-$2$ graph on $\Z_{2m}$ has two components, the even and odd cycles, each of length $m$.  The set $A$ is independent in this graph.  If $m$ is odd, the total independence number is $m-1$, so $A$ is maximum in both odd components.  In an odd cycle, a maximum independent set has one gap of two consecutive unselected vertices; the forward and backward tails are the two different ends of that gap.  Hence the two tail sets are disjoint.

If the intersection is nonempty, then $m$ is even.  Write $m=2t$.
The total independence number of the two even components is $m$, whereas
$A$ has size $m-1$.  Thus one component is a maximum alternating set of
size $t$, and the other contains $t-1$ selected and $t+1$ unselected
vertices.  In the deficient component, the $t-1$ cyclic zero-gaps between
successive selected vertices are all nonempty and contain, in total,
$t+1$ zeros.  Hence the two excess zeros occur in exactly one of two ways:
either one zero-gap has length $3$, or two zero-gaps have length $2$.

A vertex belongs to $T_2^+(A)\cap T_2^-(A)$ precisely when it is an
unselected vertex whose two step-$2$ neighbors are also unselected.  Such
a vertex exists only in the first gap pattern, and then it is the unique
middle vertex $u$ of the length-$3$ zero-gap.  Adding $u$ turns the
deficient component into a maximum alternating set; the other component
was already alternating.  Therefore, with $F=A\cup\{u\}$, translation by
$2$ exchanges selected and unselected vertices on both components, so
$F+2=F^c$.
\end{proof}

\begin{lemma}[Adjacent-row reversal lock]\label{lem:reversal-lock}
Assume Setup~\ref{setup:even-outside}.  If no Berge cycle of length $2m-1$ exists, then for every $z\in\Z_{2m}$,
\[
        e_z\cap(e_{z-1}+1)\subseteq\{z,z+1\}.
\]
\end{lemma}

\begin{proof}
If $x\in e_z\cap(e_{z-1}+1)$ and $x\notin\{z,z+1\}$, put $w=x-1\in e_{z-1}$.  The sequence
\[
        x,x+1,\ldots,z-1,w,w-1,\ldots,z+1,x
\]
uses all vertices except $z$, with chords supplied by $e_{z-1}$ and $e_z$.  This is a Berge cycle of length $2m-1$.
\end{proof}

\begin{theorem}[Next-to-Hamiltonian length]\label{thm:next-to-hamiltonian}
Let $m\ge12$.  Every setup satisfying Setup~\ref{setup:even-outside} contains a Berge cycle of length $2m-1$.
\end{theorem}
\begin{proof}
Assume no such cycle exists.  If $f$ realizes distance $2$, the chord lemma gives length $2m-1$, so $f\cap(f+2)=\varnothing$.  Since $f$ also misses distance $2$, absence of an adjacent pair would make every cyclic gap between selected vertices at least $3$, forcing $|f|\le\lfloor2m/3\rfloor<m-1$, a contradiction.  Choose $p$ with $p,p+1\in f$.

By Lemma~\ref{lem:s2-adjacent-lock}, $e_p\setminus f\subseteq T_2^+(f)\cap T_2^-(f)$.  Since $e_p\ne f$ and $|e_p|=|f|$, this set is nonempty.  Lemma~\ref{lem:s2-defect-form} gives that $m$ is even and that there is a unique vertex $u$ such that, with $F=f\cup\{u\}$, one has $F+2=F^c$.  Moreover $e_p=F\setminus\{\rho\}$ for some $\rho\in f\setminus\{p,p+1\}$.

Put $q=p+1$.  We claim that
\begin{equation}\label{eq:reverse-one-sided-lock}
        e_q\setminus f\subseteq\{q+1\}\cup T_2^-(f).
\end{equation}
Indeed, suppose that $a\in e_q\setminus f$, that $a-2\in f$, and that
$a\ne q+1$.  The case $a=q+2$ is excluded by
Lemma~\ref{lem:boundary-bypass} with $d=1$.  The case $a=q-1$ is excluded
by the other direction of the same lemma.  Rotate so that $q=0$; then we
may take $3\le a\le2m-2$.  The cyclic vertex sequence
\[
 0,\ 2m-1,\ 2m-2,\ldots,a,\ 1,\ 2,\ldots,a-2,\ 0
\]
contains every vertex except $a-1$.  Use the original Hamiltonian labels
on the two displayed arcs, use $e_0=e_q$ on the jump $a\to1$, and use $f$
on the jump $a-2\to0$.  The two arc-label intervals are disjoint, neither
contains $e_0$, and $f\notin E(C)$; hence all labels are distinct.  This is
a Berge cycle of length $2m-1$, a contradiction.  Thus
\eqref{eq:reverse-one-sided-lock} holds.

We next determine the backward tail set exactly.  Complementing
$F+2=F^c$ and using translation invariance of complements gives
\[
        F=(F+2)^c=F^c+2,
        \qquad\text{hence}\qquad F-2=F^c.
\]
Since $u\in F$, it follows that $u-2\in F^c$, while $u\notin f$; hence
$u\in T_2^-(f)$.  If $x\in T_2^-(f)$ and $x\ne u$, then
$x\in F^c=F+2$, so $x-2\in F$.  The condition $x-2\notin f=F\setminus\{u\}$
therefore forces $x-2=u$, that is, $x=u+2$.  Conversely $u+2\in F^c$
and $(u+2)-2=u\notin f$.  Thus
\[
        T_2^-(f)=\{u,u+2\}.
\]
Consequently
\begin{equation}\label{eq:next-hamiltonian-first-containment}
        e_q\subseteq f\cup\{q+1,u,u+2\}.
\end{equation}
By Lemma~\ref{lem:reversal-lock} applied to $z=q$,
\begin{equation}\label{eq:next-hamiltonian-second-containment}
        e_q\subseteq \bigl(\Z_{2m}\setminus(F+1)\bigr)
             \cup\{\rho+1,q,q+1\}.
\end{equation}
Because $F+2=F^c$, we have $F+4=F$.  Thus $F$ is a union of two residue classes modulo $4$, one even and one odd.  Since $p,p+1\in F$, these classes are adjacent, and $|F\setminus(F+1)|=m/2$.  Moreover
\[
        (F+1)+2=(F+2)+1=F^c+1=(F+1)^c.
\]
Hence translation by $2$ exchanges $F+1$ with its complement, so every pair
$\{x,x+2\}$ has exactly one member in $F+1$.  In particular, at most one
of $u,u+2$ lies outside $F+1$.

Intersecting \eqref{eq:next-hamiltonian-first-containment} and \eqref{eq:next-hamiltonian-second-containment}, the nonexceptional part contributes at most $m/2+1$ vertices: at most $m/2$ from $f\setminus(F+1)$ and at most one from $\{u,u+2\}\setminus(F+1)$, while $q+1\in F+1$.  The exceptional set $\{\rho+1,q,q+1\}$ contributes at most three more.  Therefore
\[
        |e_q|\le m/2+4.
\]
Since $|e_q|=m-1$, this gives $m\le10$, contradicting $m\ge12$.
\end{proof}

\begin{theorem}[Two-gap stability]\label{thm:two-gap-stability}
Let $A\subseteq\Z_{2m}$ satisfy $|A|=m-1$.  Suppose $1\le d\le m-2$ and
\[
        d,d+1\notin\Delta(A).
\]
Then
\[
        m\text{ is odd},\qquad d=\frac{m-1}{2},\qquad A=A+m.
\]
\end{theorem}
\begin{proof}
The assumptions say that $A-A$ avoids $T=\{\pm d,\pm(d+1)\}$.  Let $H=\Stab(A-A)$ and $|H|=h$.  By Kneser's theorem,
\[
        |A-A|\ge 2|A+H|-|H|\ge 2(m-1)-h.
\]
Hence the complement $C=\Z_{2m}\setminus(A-A)$ has size at most $h+2$.  Since $A-A$ is $H$-periodic, so is $C$, and $T+H\subseteq C$.

If $h=1$, then $|C|\le3$, impossible because $T$ has four elements.
Thus $h\ge2$.  Moreover $1\notin H$: if $1\in H$, then the subgroup $H$
is all of $\Z_{2m}$, so the nonempty $H$-periodic set $A-A$ is all of
$\Z_{2m}$, contrary to its avoidance of $T$.  Hence the cosets $d+H$ and
$d+1+H$ are distinct, and $|T+H|\ge2h$.  Since
\[
        2h\le |T+H|\le |C|\le h+2,
\]
we obtain $h=2$.  The unique subgroup of order $2$ is $H=\{0,m\}$, and
all inequalities above are equalities.  In particular,
\[
        |C|=4,
        \qquad C=(d+H)\cup(d+1+H),
        \qquad |A-A|=2m-4.
\]
Kneser's inequality now reads
\[
        2m-4=|A-A|\ge2|A+H|-2\ge2|A|-2=2m-4,
\]
so equality throughout gives $|A+H|=|A|=m-1$.
Passing to the quotient $\Z_{2m}/H\cong\Z_m$, symmetry under negation gives
\[
        \{d,d+1\}=\{-d,-d-1\}.
\]
The only elements of $\Z_m$ fixed by negation are $0$ and, when $m$ is
even, $m/2$.  Since $d$ and $d+1$ are distinct and neither is $0$ in the
stated range, they cannot both be fixed.  They must therefore be exchanged,
so
$2d+1\equiv0\pmod m$.  The range $1\le d\le m-2$ then forces $m$ odd and
$d=(m-1)/2$.

Since $A\subseteq A+H$ and the two sets have equal cardinality, $A=A+H=A+m$.
\end{proof}

For the remainder of this subsection fix an integer $s$ with
$2\le s\le m-1$, put
\[
        A:=f,
\]
and assume
\[
        |A|=m-1,
        \qquad A\cap(A+s)=\varnothing.
\]
For every row $p\in\Z_{2m}$ define
\begin{equation}\label{eq:split-GH-definitions}
        G_p:=A\cap[p+s+1,p]^+,
        \qquad
        H_p:=A\cap[p+1,p-s]^+.
\end{equation}
The sets $G_p$ and $H_p$ record, respectively, the possible endpoints of
the row-first and $f$-first split constructions.

\begin{lemma}[Long split locks]\label{lem:long-locks}
Assume no Berge cycle of length $2m-s+1$ exists.  Let $D,R\ge1$ with $D+R=s$.
\begin{enumerate}[label=\textup{(\roman*)}]
\item If $p+D\in A$, then $e_p\cap(G_p+D-s-1)=\varnothing$.
\item If $p-R\in A$, then $e_p\cap(H_p+s-R)=\varnothing$.
\end{enumerate}
\end{lemma}

\begin{proof}
For (i), suppose that $x\in e_p$ and
$x=y+D-s-1$ for some $y\in G_p$.  Since $D+R=s$, we have
$y=x+R+1\in A$, while the hypothesis gives $p+D\in A$.  After translating
$p$ to $0$, represent $y$ in the integer interval $[s+1,2m]$.  All arcs in
this proof are read in this chosen integer lift.  In particular,
$0,-1,\ldots,y$ means the descending cyclic arc
$0,2m-1,\ldots,y$; when $y=2m$, it is the singleton vertex $0$ and its
Hamiltonian label interval is empty.  Now
$x=y-R-1\in[D,2m-R-1]$.  Consequently the two Hamiltonian arcs
\[
        D,D+1,\ldots,x
        \quad\text{and}\quad
        0,-1,\ldots,y
\]
are vertex-disjoint.  Their Hamiltonian label sets are the ordinary
intervals $[D,x-1]$ and $[y,2m-1]$, respectively; by the convention in
Section~\ref{sec:prelim}, either interval is empty when its arc has one
vertex.  These intervals are disjoint and neither contains $0$.  Close the
first arc to $0$ with the free label $e_0$, and close the second arc to $D$
with the additional edge $f$.  The two special labels are
distinct from each other and from all arc labels.  The resulting cycle is
\[
        D,D+1,\ldots,x,\ 0,-1,\ldots,y,\ D
\]
and has
\[
 (x-D+1)+(2m-y+1)=2m-(D+R)+1=2m-s+1
\]
vertices, a contradiction.

For (ii), translate $p$ to $0$ and suppose that
$x\in e_0\cap(H_0+s-R)$.  Since $s-R=D$, write
\[
        x=y+D\qquad\text{with }y\in H_0.
\]
Choose the representative $y\in[1,2m-s]$.  Then
$x\in[D+1,2m-R]$.  The two vertex-disjoint Hamiltonian arcs
\[
        -R,-R-1,\ldots,x
        \qquad\text{and}\qquad
        1,2,\ldots,y
\]
have disjoint label sets: the descending arc uses the ordinary interval
$[x,2m-R-1]$, while the ascending arc uses $[1,y-1]$.  Neither interval
contains $0$, and reversed endpoint intervals are empty by convention.
Use the free label $e_0$ on
the jump $x\to1$ and the additional edge $f$ on the jump $y\to-R$; the
latter is available because $y,-R\in A$.  Thus
\[
        -R,-R-1,\ldots,x,1,2,\ldots,y,-R
\]
is a Berge cycle with
\[
        (2m-R-x+1)+y=2m-(D+R)+1=2m-s+1
\]
vertices, a contradiction.
\end{proof}

\begin{lemma}[Short split locks]\label{lem:short-locks}
Assume no Berge cycle of length $s+1$ exists.  Let $D,R\ge1$ with $D+R=s$.
\begin{enumerate}[label=\textup{(\roman*)}]
\item If $p+D\in A$, then $e_p\cap(G_p+D-s)=\varnothing$.
\item If $p-R\in A$, then $e_p\cap(H_p+s-R-1)=\varnothing$.
\end{enumerate}
\end{lemma}

\begin{proof}
For (i), suppose that $x\in e_p$ and $x=y+D-s$ with $y\in G_p$.
Then $y=x+R\in A$ and $p+D\in A$.  Translate $p$ to $0$ and represent
$y$ in $[s+1,2m]$.  All arcs in this paragraph are read in this integer
lift; in particular, $y=2m$ denotes the residue $0$, not a new vertex.
Since $x=y-R\ge D+1$, the two Hamiltonian arcs
\[
        1,2,\ldots,D
        \quad\text{and}\quad
        y,y-1,\ldots,x
\]
are vertex-disjoint.  Their label sets are the disjoint ordinary intervals
$[1,D-1]$ and $[x,y-1]$, neither of which contains $0$.  Use $f$ on the
jump $D\to y$ and use the free label $e_0$ on the jump $x\to1$.  All
labels are distinct, and the resulting cycle
\[
        1,2,\ldots,D,
        \ y,y-1,\ldots,x,
        \ 1
\]
has $D+(R+1)=s+1$ vertices, a contradiction.

For (ii), translate $p$ to $0$ and suppose that
$x\in e_0\cap(H_0+s-R-1)$.  Since $s-R=D$, write
\[
        x=y+D-1\qquad\text{with }y\in H_0.
\]
Choose $y\in[1,2m-s]$.  Then $x\le2m-R-1$.  The two vertex-disjoint
Hamiltonian arcs
\[
        0,-1,\ldots,-R
        \qquad\text{and}\qquad
        y,y+1,\ldots,x
\]
have disjoint label sets: the first uses the ordinary interval
$[2m-R,2m-1]$ and the second uses $[y,x-1]$.  Neither interval contains
$0$, and the second interval is empty when $x=y$.
Use $f$ on the jump $-R\to y$, since $-R,y\in A$, and use
the free label $e_0$ on the jump $x\to0$.  The resulting Berge cycle
\[
        0,-1,\ldots,-R,y,y+1,\ldots,x,0
\]
has $(R+1)+D=s+1$ vertices, a contradiction.
\end{proof}

\begin{lemma}[Paired interval count]\label{lem:GH-count}
For every $p$,
\[
        |G_p|+|H_p|\ge 2m-2-s.
\]
\end{lemma}

\begin{proof}
The two length-$s$ intervals
\[
        I_1=[p+1,p+s]^+,
        \qquad I_2=[p-s+1,p]^+
\]
are disjoint because $s\le m-1$, and translation by $s$ is a bijection
$I_2\to I_1$.  The assumption $A\cap(A+s)=\varnothing$ therefore implies
\[
        |A\cap(I_1\cup I_2)|\le s.
\]
By \eqref{eq:split-GH-definitions}, $G_p=A\setminus I_1$ and
$H_p=A\setminus I_2$.  Hence, using $|A|=m-1$,
\[
\begin{aligned}
 |G_p|+|H_p|
 &=2|A|-|A\cap(I_1\cup I_2)|\\
 &\ge2m-2-s.
\end{aligned}
\]
\end{proof}

\begin{proposition}[Isolated split-pair theorem]
\label{prop:bulk-split-pair}
Assume the even one-extra-edge setup.  Let $3\le s\le m-4$ and suppose
$f\cap(f+s)=\varnothing$.
\begin{enumerate}[label=\textup{(\roman*)}]
\item If $s+1\in\Delta(f)$, then the setup contains a Berge cycle of
length $2m-s+1$.
\item If $s-1\in\Delta(f)$, then the setup contains a Berge cycle of
length $s+1$.
\end{enumerate}
\end{proposition}

\begin{proof}
Put $A=f$.  First suppose $s+1\in\Delta(A)$.  Choose $a,b\in A$ with
$b=a+s+1$ and put
\[
        p=b-2=a+s-1.
\]
Then $p+2=b\in A$ and $p-(s-1)=a\in A$.  If length $2m-s+1$ were absent,
apply the row-first long lock with the legal split
\[
        (D,R)=(2,s-2)
\]
and the $f$-first long lock with the different legal split
\[
        (D,R)=(1,s-1).
\]
These two applications forbid, respectively,
\[
        G_p-(s-1),
        \qquad
        H_p+1.
\]
The two sets are disjoint: an equality
$y-(s-1)=z+1$ with $y,z\in A$ would give $y=z+s$, contrary to
$A\cap(A+s)=\varnothing$.  Hence at least
\[
        |G_p|+|H_p|\ge2m-2-s\ge m+2
\]
vertices are forbidden for $e_p$, which is impossible because
$|e_p|=m-1$.

Now suppose $s-1\in\Delta(A)$.  Choose $a,b\in A$ with
$b=a+s-1$ and put
\[
        p=b-1=a+s-2.
\]
Then $p+1=b\in A$ and $p-(s-2)=a\in A$.  If length $s+1$ were absent,
apply the row-first short lock with
\[
        (D,R)=(1,s-1)
\]
and the $f$-first short lock with
\[
        (D,R)=(2,s-2).
\]
They again forbid the disjoint sets
\[
        G_p-(s-1),
        \qquad
        H_p+1,
\]
and the same count gives a contradiction.
\end{proof}

\begin{theorem}[Antipodal two-block lift]\label{thm:antipodal-two-block}
Let $m=2r+1$ with $r\ge5$, and let $n=2m=4r+2$.  In Setup~\ref{setup:even-outside}, suppose $f=f+m$ and suppose that $f$ misses the two cyclic distances $r$ and $r+1$.  Then the setup contains Berge cycles of lengths
\[
        r+1,
        \quad r+2,
        \quad 3r+2,
        \quad 3r+3.
\]
\end{theorem}
\begin{proof}
Let
\[
 \pi:\Z_{4r+2}\longrightarrow\Z_{2r+1}
\]
be the quotient map modulo the antipodal subgroup
$\{0,2r+1\}$.  Since $f=f+(2r+1)$, there is a set
$B\subseteq\Z_{2r+1}$ such that
\[
        f=\pi^{-1}(B).
\]
Every quotient point has two lifts, so $|B|=|f|/2=r$.

We claim that $B$ misses quotient distance $r$.  Indeed, if
$b,b+r\in B$, choose a lift $x\in\pi^{-1}(b)$.  The two lifts of $b+r$
are $x+r$ and $x+r+(2r+1)$.  Hence the full lift $f$ contains one pair at
ambient cyclic distance $r$ and another at ambient cyclic distance
\[
 \min\{3r+1,(4r+2)-(3r+1)\}=r+1,
\]
contrary to the hypothesis that both distances are missing.  Thus
$r\notin\Delta(B)$.

Lemma~\ref{lem:odd-distance} now applies to $B$.  Since
$2r\equiv-1\pmod{2r+1}$, a maximum independent set in the $r$-step cycle
is, after translation, the interval $[0,r-1]$.  Rotating the ambient cyclic
order by the same translation gives
\[
        B=[0,r-1],
        \qquad
        f=\pi^{-1}(B)=[0,r-1]\cup[2r+1,3r]
        \subseteq\Z_{4r+2}.
\]
The neighbor hypotheses in Proposition~\ref{prop:bulk-split-pair} are
explicit in this normal form.  Since $\{0,r-1\}\subseteq f$, one has
$r-1\in\Delta(f)$; applying part~(ii) with the missing distance $s=r$
gives length $r+1$.  Also $\{0,3r\}\subseteq f$, and this pair has ambient
cyclic distance
\[
        \min\{3r,(4r+2)-3r\}=r+2.
\]
Thus $r+2\in\Delta(f)$; applying part~(i) with the missing distance
$s=r+1$ gives length $3r+2$.

It remains to prove the cross lengths.  For length $3r+3$, take $s=r$ and $p=4r-1$.  Then
\[
        H_p=[0,r-1]\cup[2r+1,3r-1],
        \qquad |H_p|=2r-1.
\]
Since $p-(r-1)=3r\in f$, the closing long lock forbids $H_p+1$.  Also $G_p$ contains $[2r+1,3r]$, and since $p+3\equiv0\in f$, the row long lock forbids
\[
        [2r+1,3r]-(r-2)=[r+3,2r+2].
\]
The interval $[r+3,2r+1]$ is disjoint from $H_p+1$ and has size $r-1$.  Thus the forbidden set has size at least $(2r-1)+(r-1)=3r-2\ge2r+3=m+2$, impossible if length $3r+3$ were absent.

For length $r+2$, take $s=r+1$ and $p=4r$.  Then
\[
        H_p=[0,r-1]\cup[2r+1,3r-1],
        \qquad |H_p|=2r-1.
\]
Since $p-r=3r\in f$, the closing short lock forbids $H_p$.  Also $G_p=[2r+1,3r]$, and since $p+2\equiv0\in f$, the row short lock forbids
\[
        G_p-(r-1)=[r+2,2r+1].
\]
The interval $[r+2,2r]$ is disjoint from $H_p$ and has size $r-1$.  Again the forbidden set has size at least $3r-2\ge m+2$, a contradiction if length $r+2$ were absent.
\end{proof}

\begin{theorem}[Lengths outside the middle band]\label{thm:outside-middle-band}
Every even one-extra-edge setup with $m\ge12$ contains Berge cycles of
lengths $2$, $3$, $2m-1$,
$2m$, and every length $L$ whose associated distance
\[
        s(L)=\min\{L-1,2m-L+1\}
\]
satisfies $3\le s(L)\le m-4$.
\end{theorem}
\begin{proof}
The endpoint lengths are supplied by Lemmas~\ref{lem:even-two-cycle} and~\ref{lem:even-triangle}, Theorem~\ref{thm:next-to-hamiltonian}, and the
Hamiltonian cycle.  Fix $s\in[3,m-4]$.  If $s$ is realized by $f$, the
chord lemma gives both $s+1$ and $2m-s+1$.  Suppose that $f$ misses $s$.
If both neighboring distances $s-1$ and $s+1$ are realized, apply
Proposition~\ref{prop:bulk-split-pair}.  Otherwise $f$
misses two consecutive interior distances.  Theorem~\ref{thm:two-gap-stability}
then gives $m=2r+1$, antipodal invariance $f=f+m$, and the exceptional
pair $r,r+1$.  Theorem~\ref{thm:antipodal-two-block} supplies all four
cycle lengths associated with those two distances.  Hence both lengths
associated with $s$ occur.
\end{proof}

\section{The central middle length}\label{sec:central-length}

The central case is reduced first to antipodal-free Hamiltonian labels.  If
the extra edge $f$ contains an antipodal pair, the chord lemma gives length
$m+1$.  If some Hamiltonian label $e_p$ contains an antipodal pair
$\{a,a+m\}$, choose whichever of the two half-circle index intervals
$[a,a+m-1]^+$ and $[a+m,a-1]^+$ contains $p$.  Use $e_p$ on the bridge
$\{a,a+m\}$, and choose a one-index gap disjoint from that half-circle
interval, using its original Hamiltonian label on the adjacent bridge.
The two bridge labels are distinct and the two-gap certificate again gives
length $m+1$.  Thus a counterexample to the
central length satisfies Setup~\ref{setup:central}.

\begin{setup}[Central spin setup]\label{setup:central}
Let $n=2m$.  A \emph{central antipodal-free Hamiltonian label system} is a collection of distinct $(m-1)$-sets $e_0,\ldots,e_{2m-1}\subseteq\Z_{2m}$ such that $\{p,p+1\}\subseteq e_p$ and no $e_p$ contains an antipodal pair.
For $1\le k\le m-2$, define
\[
\rho_p(k)=
\begin{cases}
+,&p-k\in e_p,\\
-,&p-k+m\in e_p,\\
0,&\text{neither }p-k\text{ nor }p-k+m\text{ lies in }e_p.
\end{cases}
\]
The anchors $p,p+1$ exclude their antipodes.  The remaining vertices form
$m-2$ antipodal pairs, and the $(m-1)$-set $e_p$ selects one point from
exactly $m-3$ of them.  Thus each row has exactly one zero.
\end{setup}

For each column $k$, set
\[
        P_k=\{p:\rho_p(k)=+\},\qquad
        M_k=\{p:\rho_p(k)=-\},\qquad
        Z_k=\{p:\rho_p(k)=0\}.
\]
The \emph{support} of column $k$ is its nonzero-row set
\[
        \operatorname{supp}(k):=P_k\cup M_k
        =\Z_{2m}\setminus Z_k.
\]
We say that column $k$ has \emph{full support} when $Z_k=\varnothing$.
\subsection{Long-range triggers}
\label{subsec:long-range-central-triggers}

Retain the central spin setup.  Thus
\[
 \rho_p(k)=+
 \quad\Longleftrightarrow\quad p-k\in e_p,
 \qquad
 \rho_p(k)=-
 \quad\Longleftrightarrow\quad p+m-k\in e_p,
\]
for $1\le k\le m-2$, and each row has exactly one zero.

The following two triggers exploit the full range of half-circle placements
and drive the structural extraction.

\begin{lemma}[Half-circle opposition trigger]
\label{lem:half-circle-opposition}
Let $1\le k\le m-2$.  If
\[
        \rho_p(k)=+,
        \qquad
        \rho_{p+s}(k)=-
\]
for some $1\le s\le m$, then the Hamiltonian labels contain a free-label
two-gap Berge cycle of length $m+1$.
\end{lemma}

\begin{proof}
The plus cell says that $e_p$ covers
\[
        \{p-k,p+1\}.
\]
Use the free interval
\[
        I=[p-k,p]^{+},
        \qquad |I|=k+1.
\]
The minus cell says that $e_{p+s}$ covers
\[
        \{p+s,p+s+m-k\}.
\]
Use
\[
        J=[p+s,p+s+m-k-1]^{+},
        \qquad |J|=m-k.
\]
After translating $p$ to $0$, the first interval is
\[
        \{2m-k,\ldots,2m-1,0\},
\]
while the second is contained in
\[
        \{1,\ldots,2m-k-1\}
\]
for $1\le s\le m$.  Hence $I$ and $J$ are disjoint.  Their lengths add to
$m+1$, the labels $p$ and $p+s$ lie in the respective free intervals, and
they are distinct.  Proposition~\ref{prop:free-label} applies.
\end{proof}

For a column $k$, put
\[
        \bar k=m-1-k.
\]

\begin{lemma}[Reflected-column trigger]
\label{lem:reflected-column-trigger}
Let $1\le k\le m-2$ and let $\bar k=m-1-k$.  Suppose that two cells in
columns $k$ and $\bar k$ have the same nonzero sign.  If their row separation
$s$ belongs to
\[
        B_k=[m-k,\,2m-k-1]^{+},
\]
then the labels contain a free-label two-gap Berge cycle of length $m+1$.
The interval $B_k$ contains exactly $m$ row separations.
\end{lemma}

\begin{proof}
For two plus cells at rows $p$ and $p+s$, use
\[
 I=[p-k,p]^{+},
 \qquad
 J=[p+s-m+1+k,p+s]^{+}.
\]
Their lengths are $k+1$ and $m-k$.  After translating $p$ to $0$, the two
intervals are disjoint precisely for
\[
        m-k\le s\le2m-k-1.
\]
The corresponding bridges are covered by $e_p$ and $e_{p+s}$.

For two minus cells use
\[
 I=[p,p+m-k-1]^{+},
 \qquad
 J=[p+s,p+s+k]^{+}.
\]
The same separation range makes the intervals disjoint, and their lengths
are again $m-k$ and $k+1$.  In both cases the bridge labels are free and
distinct, so Proposition~\ref{prop:free-label} applies.
\end{proof}

For a column $k$, put $z_k=|Z_k|$.  Call $k$ \emph{mixed} if both
$P_k$ and $M_k$ are nonempty.

\begin{lemma}[Zero budget of a mixed column]
\label{lem:mixed-column-zero-budget}
If there is no free-label two-gap cycle of length $m+1$ and column $k$ is
mixed, then
\[
        z_k\ge m.
\]
Consequently at most two columns are mixed.
\end{lemma}

\begin{proof}
Read the nonzero cells in column $k$ cyclically.  Since the column is
mixed, there are two consecutive cells in this cyclic nonzero sequence
whose signs change from plus to minus.  If there are $z$ zero cells strictly
between those two cells in the full column, their clockwise row separation
is exactly $z+1$.  By
Lemma~\ref{lem:half-circle-opposition}, this separation is greater than
$m$.  Thus $z\ge m$.  Since every one of the $2m$ rows has exactly one zero,
\[
        \sum_{k=1}^{m-2}z_k=2m,
\]
and at most two columns can be mixed.
\end{proof}

\begin{lemma}[Monochromatic reflected columns have opposite signs]
\label{lem:reflected-opposite-signs}
Assume there is no free-label two-gap cycle of length $m+1$.  Let $k\ne\bar
k$.  If columns $k$ and $\bar k$ are both nonmixed and both contain a
nonzero cell, then their nonzero signs are opposite.
\end{lemma}

\begin{proof}
Suppose that both columns have the same sign.  Fix a nonzero row $p$ in
column $k$.  Lemma~\ref{lem:reflected-column-trigger} implies that the
support of column $\bar k$ avoids the $m$-set $p+B_k$.  Hence
$z_{\bar k}\ge m$.  Interchanging the two columns gives $z_k\ge m$.
Therefore
\[
 z_k=z_{\bar k}=m
\]
and every other column has no zeros.

The support of column $\bar k$ has size $m$ and is contained in the
complement of $p+B_k$, which is also an interval of size $m$.  Thus equality
holds.  Repeating the argument with a second support point $p'$ of column
$k$ says that the same length-$m$ interval is invariant under translation
by $p'-p$.  A proper cyclic interval of length $m$ in $\Z_{2m}$ has
trivial stabilizer: a stabilizing translation must preserve the set of its
two boundary transitions; fixing either transition forces the zero
translation, while exchanging the two transitions is translation by $m$,
which sends the interval to its complement.  Since the support of column
$k$ has $m\ge2$ points, this is impossible.
\end{proof}

\begin{lemma}[Two mixed reflected columns are impossible]
\label{lem:two-mixed-reflected-impossible}
Assume there is no free-label two-gap cycle of length $m+1$.  If two columns
are mixed, then they cannot be a reflected pair.
\end{lemma}

\begin{proof}
Suppose $k$ and $\bar k$ are mixed.  By
Lemma~\ref{lem:mixed-column-zero-budget},
\[
        z_k=z_{\bar k}=m,
\]
and their zero sets partition the rows.

A mixed column having exactly $m$ zeros has a rigid form.  Choose a plus
cell followed cyclically by the next minus cell.  The proof of
Lemma~\ref{lem:mixed-column-zero-budget} shows that all $m$ zeros lie between
these two cells.  On the remaining length-$m$ interval there can be no
further plus-to-minus transition.  Thus, after rotation, the column is
\[
        0^m\,-^a\,+^{m-a}
\]
for some $1\le a\le m-1$.

Normalize so that the zero interval of column $k$ is
$\{1,\ldots,m\}$.  Then row $m+1$ of column $k$ is minus.  The zero interval
of column $\bar k$ is its complement, so row $1$ of column $\bar k$ is also
minus.  Their row separation is $m$, and
\[
        m\in[m-k,2m-k-1].
\]
Lemma~\ref{lem:reflected-column-trigger} gives the forbidden cycle.
\end{proof}

\subsection{Exact reflected-wave extraction}
\label{subsec:exact-reflected-extraction}

\begin{theorem}[Exact reflected-wave extraction]
\label{thm:exact-reflected-extraction}
Assume a central antipodal-free Hamiltonian label system has no free-label
two-gap Berge cycle of length $m+1$.

\begin{enumerate}[label=\textup{(\roman*)}]
\item If $m$ is even, there are signs
\[
        \sigma_k\in\{+,-\},
        \qquad \sigma_{\bar k}=-\sigma_k,
\]
such that every nonzero spin cell in column $k$ has sign $\sigma_k$.

\item If $m=2h+1$ is odd, the same conclusion holds on the noncentral
columns $k\ne h$.  The central column $h$ is unrestricted.
\end{enumerate}
There is no exceptional noncentral cell.
\end{theorem}
\begin{proof}
First we show that no two columns can be mixed.  Their zero counts would
be at least $m$ each, so every other column would have full support.  At
least one of the two mixed columns is noncentral; choose such a column
$k$.  If its mate $\bar k$ is the other mixed column, then
Lemma~\ref{lem:two-mixed-reflected-impossible} applies.  Otherwise
$\bar k$ is one of the full-support columns and has a fixed nonzero sign.
The mixed column $k$ contains that same sign, and a row of $\bar k$ at a
separation in $B_k$ supplies the reflected-column trigger.  This is again a
contradiction.  Hence at most one column is mixed.  In particular, there is
at most one mixed noncentral column.

Suppose that $k$ is the unique mixed noncentral column.  Its mate $\bar k$
is nonmixed.  Moreover it is not all zero: by
Lemma~\ref{lem:mixed-column-zero-budget}, $z_k\ge m$, whereas an all-zero
mate would have $z_{\bar k}=2m$, contradicting
$\sum_\ell z_\ell=2m$.  Thus $\bar k$ has a unique nonzero sign $\tau$.
Choose a row $p$ in column $k$ with sign $\tau$.  The reflected-column trigger forces the support of
$\bar k$ into the complement of the length-$m$ interval $p+B_k$.  Hence
$z_{\bar k}\ge m$.  Together with
Lemma~\ref{lem:mixed-column-zero-budget}, this gives
\[
        z_k=z_{\bar k}=m,
\]
and every other column has no zero.  The support of $\bar k$ has size $m$
and is therefore exactly the complement of $p+B_k$.

Every row has exactly one zero, so the zero sets of $k$ and $\bar k$
partition the rows.  Since $(p,k)$ is nonzero, one must have
$p\in Z_{\bar k}$.  On the other hand $0\notin B_k$, so
$p\notin p+B_k$ and hence $p$ belongs to the complement of $p+B_k$, which
is the support of $\bar k$.  This contradiction proves that no noncentral
column is mixed.  In odd parity it also rules out a noncentral mixed column
coexisting with a mixed central column, since the former's mate would have
full support.

Assign to every nonzero relevant column its unique sign.
Lemma~\ref{lem:reflected-opposite-signs} makes the signs opposite on each
reflected pair.  If one member of a reflected pair is all zero, its mate
cannot also be all zero, since two all-zero columns would already contain
$4m$ zero cells while the whole array contains only $2m$.  The mate
therefore has a defined nonzero sign; choose the base sign of the all-zero
column to be its opposite.  For odd $m$ omit the fixed central column.
This proves both parts.
\end{proof}

The signs in Theorem~\ref{thm:exact-reflected-extraction} define a
reflected translation wave.  If $\sigma_k=+$ include the offset $-k$, and
if $\sigma_k=-$ include $m-k$.

\begin{definition}[Reflected base waves]\label{def:reflected-base-waves}
A set $E\subseteq\Z_{2m}$ is a \emph{full complementary-transversal wave}
if
\[
 E=1-E,
 \qquad E\cap(E+m)=\varnothing,
 \qquad E\cup(E+m)=\Z_{2m},
 \qquad \{0,1\}\subseteq E.
\]
A set $E\subseteq\Z_{2m}$ is a \emph{zero-free reflected wave} if
\[
 E=1-E,
 \qquad E\cap(E+m)=\varnothing,
 \qquad \{0,1\}\subseteq E,
\]
and one specified antipodal pair is omitted:
\[
 \Z_{2m}\setminus\bigl(E\cup(E+m)\bigr)=\{\xi,\xi+m\}
\]
for some $\xi$.  When $m$ is odd, the omitted pair is the
reflection-invariant pair represented by the central spin column.
\end{definition}

\begin{corollary}[Exact reflected-wave deletion normal form]
\label{cor:one-extra-deletion-normal-form}
Under the hypotheses of Theorem~\ref{thm:exact-reflected-extraction}, there
is a reflection-symmetric base set $E$ with the following properties.

\begin{enumerate}[label=\textup{(\roman*)}]
\item If $m$ is even, $|E|=m$, $E$ is a full
complementary-transversal wave, and every translated row satisfies
\[
        e_p-p=E\setminus R_p,
        \qquad |R_p|=1.
\]

\item If $m=2h+1$ is odd, $|E|=m-1$ and $E$ is the zero-free reflected
wave built from one representative of every noncentral reflected pair.
For each row $p$, let $R_p=\varnothing$ when the central column is zero,
and otherwise let $R_p$ be the singleton consisting of the deleted
noncentral base offset.  Then $|R_p|\le1$, and the row is obtained from
$E\setminus R_p$ by possibly inserting its central-column offset.
\end{enumerate}
Inserted offsets can only create bridge covers, and in either parity
\[
        \sum_{p\in\Z_{2m}}|R_p|\le2m.
\]
\end{corollary}

\begin{proof}
For $1\le k\le m-2$, let
\[
        \Omega_k:=\{-k,m-k\}
\]
be the antipodal offset pair represented by spin column $k$.  These
$m-2$ pairs are distinct and are precisely the antipodal pairs other than
the two anchor pairs $\{0,m\}$ and $\{1,m+1\}$.  The reflection
$\tau(x)=1-x$ sends $\Omega_k$ to $\Omega_{\bar k}$, because
\[
        \tau(-k)=k+1=m-\bar k,
        \qquad
        \tau(m-k)=k+1-m=-\bar k.
\]
For every relevant column define its selected base offset by
\[
 b_k=
 \begin{cases}
 -k,&\sigma_k=+,\\
 m-k,&\sigma_k=-.
 \end{cases}
\]
The relation $\sigma_{\bar k}=-\sigma_k$ gives
$b_{\bar k}=1-b_k$.

If $m$ is even, set
\[
        E:=\{0,1\}\cup\{b_k:1\le k\le m-2\}.
\]
This set contains exactly one point from every antipodal pair, has size
$m$, and satisfies $E=1-E$.  Hence it is a full
complementary-transversal wave.  By
Theorem~\ref{thm:exact-reflected-extraction}, every nonzero cell in column
$k$ selects the offset $b_k$.  Each row has exactly one zero, so its
translated edge contains all elements of $E$ except the single base offset
belonging to that zero column.  This proves part~(i).

Now let $m=2h+1$ and omit the fixed central column $h=\bar h$.  Put
\[
        E:=\{0,1\}\cup\{b_k:1\le k\le m-2,\ k\ne h\}.
\]
The noncentral columns represent all nonanchor antipodal pairs except
$\Omega_h$.  Thus $|E|=m-1$, $E$ is antipodal-free, and
\[
 \Z_{2m}\setminus\bigl(E\cup(E+m)\bigr)=\Omega_h.
\]
The pair $\Omega_h$ is reflection-invariant, while the opposite-sign rule
on every other reflected pair gives $E=1-E$.  Hence $E$ is the asserted
zero-free reflected wave.  If a row is zero in the central column, then its
unique zero occurs there and every noncentral column contributes its base
offset, so the translated row is exactly $E$ and $R_p=\varnothing$.  If
the central cell is nonzero, it contributes one offset from the omitted
pair $\Omega_h$; the row's unique zero then lies in a noncentral column and
deletes exactly that column's base offset from $E$.  This deleted offset is
the unique element of $R_p$.  Inserted central offsets can only create
additional bridge covers.  In both parities $|R_p|\le1$ for every one of
the $2m$ rows, and therefore
$\sum_p|R_p|\le2m$.
\end{proof}

\section{Cover surplus for reflected waves}\label{sec:cover-surplus}

We now prove that the reflected-wave normal form obtained in
Corollary~\ref{cor:one-extra-deletion-normal-form} always contains a
surviving two-gap template when $m$ is sufficiently large.

\subsection{Bridge candidates and deletion degrees}

Let $E\subseteq\Z_{2m}$ satisfy $E=1-E$ and contain $0,1$.  For
$1\le d\le m$ define
\[
 C_d=\{i\in[0,m]:\{-i,d-i\}\subseteq E\},
\]
\[
 D_d=\{j\in[0,m]:\{d-j,m+1-j\}\subseteq E\}.
\]
For a rotation $x$, an ordered pair $(i,j)\in C_d\times D_d$ with
$i\ne j$ uses the labels $e_{x+i}$ and $e_{x+j}$ on the bridges
$\{x,x+d\}$ and $\{x+d,x+m+1\}$, respectively.  The corresponding free
intervals are $[x,x+d-1]^{+}$ and $[x+d,x+m]^{+}$, so the two-gap certificate
produces length $m+1$.  Put
\[
 Q_d=|C_d||D_d|-|C_d\cap D_d|,
 \qquad Q(E)=\sum_{d=1}^{m}Q_d.
\]
Thus there are $2mQ(E)$ rotated bridge candidates.

For $\alpha\in E\setminus\{0,1\}$ define
\[
\begin{aligned}
 \operatorname{del}_E(\alpha)=\sum_{d=1}^{m}
 &\sum_{\substack{i\in C_d\\ \alpha\in\{-i,d-i\}}}
 \bigl(|D_d|-\1_{i\in D_d}\bigr)\\
 +\sum_{d=1}^{m}
 &\sum_{\substack{j\in D_d\\ \alpha\in\{d-j,m+1-j\}}}
 \bigl(|C_d|-\1_{j\in C_d}\bigr),
\end{aligned}
\]
and put
\[
        \Delta_{\mathrm{del}}(E)
        :=\max_{\alpha\in E\setminus\{0,1\}}\operatorname{del}_E(\alpha).
\]

\begin{lemma}[Global candidate survival]\label{lem:global-candidate-survival}
For every $p\in\Z_{2m}$, let $R_p\subseteq E\setminus\{0,1\}$ and assume
\[
        E\setminus R_p\subseteq e_p-p;
\]
arbitrary additional offsets may also lie in $e_p-p$.  If
\[
        \sum_{p\in\Z_{2m}}|R_p|\le2m
        \qquad\text{and}\qquad Q(E)>\Delta_{\mathrm{del}}(E),
\]
then the labels contain a free-label two-gap Berge cycle of length $m+1$.
\end{lemma}

\begin{proof}
For a first-bridge deletion in row $p$, a candidate with first-cover
index $i$ has the uniquely determined rotation $x=p-i$; for a
second-bridge deletion with second-cover index $j$, the rotation is
uniquely $x=p-j$.  Therefore a deletion of the offset $\alpha$ in a fixed
row kills at most
$\operatorname{del}_E(\alpha)\le\Delta_{\mathrm{del}}(E)$ rotated candidates, counted with
multiplicity.  Hence all row deletions kill at most $2m\Delta_{\mathrm{del}}(E)$
candidates.  Since the total supply is $2mQ(E)$, at least one candidate
survives.  Inserted offsets can only create further covers.
\end{proof}

Put
\[
        I=\{i\in[0,m]:-i\in E\},
        \qquad c_d=|C_d|.
\]

\begin{lemma}[Folded representative structure]
\label{lem:folded-representative-structure}
For either reflected-wave type in
Definition~\ref{def:reflected-base-waves}, one has
\[
        0\in I,
        \qquad m-1,m\notin I,
        \qquad E=(-I)\mathbin{\dot\cup}(I+1).
\]
If
\[
        \mathcal A:=I\setminus\{0\},
        \qquad \Pi_a:=\{-a,a+1\}\quad(a\in\mathcal A),
\]
then
\[
        E=\{0,1\}\mathbin{\dot\cup}
          \mathop{\dot\bigcup}_{a\in\mathcal A}\Pi_a.
\]
In particular, every deletable nonanchor base offset belongs to a unique
pair $\Pi_a$, and
\[
        E_{I\setminus\{a\}}=E\setminus\Pi_a
        \qquad(a\in\mathcal A),
\]
where $E_J=(-J)\cup(J+1)$.
\end{lemma}

\begin{proof}
The anchor $0\in E$ gives $0\in I$.  Since $E$ is antipodal-free and
contains $0,1$, neither $m$ nor $m+1$ belongs to $E$.  Therefore
$m\notin I$ and $m-1\notin I$.

If $x\in E\cap\{0,m,m+1,\ldots,2m-1\}$, then $x=-i$ for a unique
$i\in[0,m]$, and hence $i\in I$.  If instead
$x\in E\cap[1,m-1]$, reflection gives $1-x\in E$, so
$x=(x-1)+1$ with $x-1\in I$.  Thus $E=(-I)\cup(I+1)$.
If $-i=j+1$ for $i,j\in I$, then $i+j=2m-1$.  Because
$i,j\in[0,m]$, this would force
$\{i,j\}=\{m-1,m\}$, which is impossible.  The union is therefore
disjoint.

Removing $a$ from $I$ removes exactly the two uniquely represented offsets
$-a$ and $a+1$, proving all remaining assertions.
\end{proof}

\begin{lemma}[Profile folding]\label{lem:profile-folding}
For every $d\in[1,m]$, the map $j\mapsto m-j$ is a bijection
$D_d\to C_{m+1-d}$.  Consequently
\[
 P(E):=\sum_{d=1}^{m}|C_d||D_d|
      =\sum_{d=1}^{m}c_d c_{m+1-d}.
\]
For both reflected-wave types arising in
Corollary~\ref{cor:one-extra-deletion-normal-form},
\[
        Q(E)\ge\frac12P(E).
\]
\end{lemma}

\begin{proof}
If $j\in D_d$ and $i=m-j$, then, using $E=1-E$,
\[
        -i=1-(m+1-j)\in E,
        \qquad m+1-d-i=1-(d-j)\in E.
\]
Thus $i\in C_{m+1-d}$.  Conversely, suppose that
$i\in C_{m+1-d}$ and put $j=m-i$.  Then
\[
        -i\in E,
        \qquad m+1-d-i\in E.
\]
Using $E=1-E$, we obtain
\[
        m+1-j=i+1=1-(-i)\in E
\]
and
\[
        d-j=d-m+i=1-(m+1-d-i)\in E.
\]
Hence $j\in D_d$.  Since $j\mapsto m-j$ is its own inverse, it is a
bijection from $D_d$ to $C_{m+1-d}$, proving the formula for $P$.

We prove the cellwise inequality
\[
        2|C_d\cap D_d|\le |C_d||D_d|.
\]
First, $C_m=\varnothing$, because membership of $i$ in $C_m$ would require
$E$ to contain the antipodal pair $\{-i,m-i\}$.  By the bijection just
proved, $D_1=\varnothing$.  Hence the endpoint cells cause no difficulty.

Suppose first that $E$ is a full wave and $2\le d\le m-1$.  Put
$\bar d=m+1-d$.  Exactly one of $d$ and $\bar d$ lies in $E$.  Indeed, if
$d\in E$, then $1-d\in E$ by reflection, so its antipode $\bar d$ is not
in $E$.  Conversely, if $d\notin E$, then $d+m\in E$ by antipodal
transversality, and reflection gives $\bar d\in E$.  If $d\in E$, then
\[
        0,d-1\in C_d,
\]
because $\{0,d\}\subseteq E$ and
$\{1-d,1\}\subseteq E$.  Thus $|C_d|\ge2$.  If $\bar d\in E$, the same
argument gives $|C_{\bar d}|\ge2$.  Since
$|D_d|=|C_{\bar d}|$, the exceptional case in which both factors equal
one cannot occur.

Now let $E$ be a zero-free reflected wave, and write its omitted antipodal
pair as $\{d_0,d_0+m\}$ with $d_0\in[1,m]$.  The pair is invariant under the
reflection $x\mapsto1-x$.  That reflection cannot fix either member, since
$2x\equiv1\pmod{2m}$ has no solution.  Therefore
\[
        1-d_0\equiv d_0+m\pmod{2m},
        \qquad 2d_0\equiv1-m\pmod{2m}.
\]
The latter congruence forces $m$ to be odd and, for $d_0\in[1,m]$, gives
\[
        d_0=\frac{m+1}{2}.
\]
For $d\ne d_0$, the preceding full-wave argument is unchanged.  At $d=d_0$ one has $m+1-d=d$, and the bijection
$D_d\to C_d$ is $j\mapsto m-j$.  This involution has no fixed point because
$m$ is odd.  Therefore, if both $C_d$ and $D_d$ are singletons, they are
disjoint, so $C_d\cap D_d=\varnothing$.

The cellwise inequality follows in every case: outside the singleton case
it is immediate from
$|C_d\cap D_d|\le\min\{|C_d|,|D_d|\}$.  Summing over $d$ yields
$Q(E)\ge P(E)/2$.
\end{proof}

\subsection{A folded cover-mass theorem}

Fix $\epsilon\in\{1,2\}$, put
\[
        N=2H+\epsilon,
        \qquad T=N+2,
\]
and let $r\in[H]$.  For a partition
\[
        L\sqcup U=[H]\setminus\{r\}
\]
define
\[
        J=\{0\}\cup L\cup\{N-u:u\in U\}.
\]
For $d\ge1$ set
\[
\begin{aligned}
 A_d(J)&=\#\{(x,y)\in J^2:x-y=d\},\\
 B_d(J)&=\#\{(x,y)\in J^2:x+y=d-1\}.
\end{aligned}
\]
and
\[
 P_N(J)=\sum_{d=1}^{N+1}
        \bigl(A_d(J)+B_d(J)\bigr)
        \bigl(A_{T-d}(J)+B_{T-d}(J)\bigr).
\]
With $N,J$ fixed, write
\[
\begin{aligned}
 AA(J)&=\sum_{d=1}^{N+1}A_d(J)A_{T-d}(J),&
 AB(J)&=\sum_{d=1}^{N+1}A_d(J)B_{T-d}(J),\\
 BA(J)&=\sum_{d=1}^{N+1}B_d(J)A_{T-d}(J),&
 BB(J)&=\sum_{d=1}^{N+1}B_d(J)B_{T-d}(J).
\end{aligned}
\]
The substitution $d\mapsto T-d$ gives $BA(J)=AB(J)$, and therefore
\[
        P_N(J)=AA(J)+2AB(J)+BB(J).
\]
For $W\subseteq[H]$ write
\[
        \mathsf T(W)=\#\{(x,y)\in W^2:x+y+1\in W\}.
\]

We use the following stability consequence, recorded as Lemma~2.2
in~\cite{BLST}.  As noted there, it follows from the structure theorem of
Deshouillers, Freiman, S\'os and Temkin~\cite{DFST} together with Green's
removal lemma~\cite{GreenRemoval}.

\begin{lemma}[Schur-triple stability near density $1/2$]
\label{lem:schur-stability-consequence}
Let $V_n\subseteq[n]$ satisfy
\[
        |V_n|\ge n/2-o(n)
\]
and contain $o(n^2)$ ordered solutions of $x+y=z$.  Put
\[
 O_n=\{1\le x\le n:x\text{ is odd}\},
 \qquad U_n=[\lfloor n/2\rfloor+1,n].
\]
Then
\[
        \min\bigl\{|V_n\triangle O_n|,\ |V_n\triangle U_n|\bigr\}=o(n).
\]
The closer model may depend on $n$.
\end{lemma}

\begin{proof}
The convention in \cite{BLST} may identify the two orders of the summands
in a Schur triple.  Ordered and unordered conventions differ by at most a
factor of two (with only $O(n)$ diagonal solutions), so the hypothesis
$o(n^2)$ is unchanged.  Lemma~2.2 of \cite{BLST} has three alternatives.
Its small-set alternative
$|V_n|\le0.47n$ is excluded.  In each of the other two alternatives,
\[
        |V_n|=(1/2-\gamma)n,
        \qquad -o(1)\le\gamma\le0.03.
\]
Our lower bound $|V_n|\ge n/2-o(n)$ implies $\gamma\le o(1)$; together
with the displayed lower bound, this gives $\gamma=o(1)$.  For each $n$,
apply the alternative supplied by the lemma.  In the odd
alternative, all but $o(n)$ elements lie among the odd integers, and the
size assumption forces only $o(n)$ odd integers to be missing.  In the
interval alternative, all but $o(n)$ elements lie in
$[(1/2-\gamma)n,n]=[(1/2-o(1))n,n]$.  This interval differs from the upper
half in only $o(n)$ positions, and the size condition forces only $o(n)$
upper-half elements to be missing.  Thus at least one of the two symmetric
differences is $o(n)$ for each $n$, which is exactly the stated minimum
formulation.
\end{proof}

\begin{lemma}[Upper Schur supply]\label{lem:upper-schur-supply}
For both values of $\epsilon$,
\[
        P_N(J)\ge4\mathsf T(U).
\]
\end{lemma}

\begin{proof}
Fix an ordered triple $(x,y,z)\in U^3$ with $x+y+1=z$.  The upper
representatives $N-x$ and $N-z$ form the positive-difference pair
\[
        (N-x)-(N-z)=z-x=y+1.
\]
At the complementary profile index, the ordered pairs
$(0,N-y)$ and $(N-y,0)$ both have the required sum $N-y$.  Hence this
triple contributes two mixed-correlation quadruples.  The construction is
injective: the positive-difference pair recovers $x,z$, and the nonzero
entry in the zero--upper sum recovers $y$.  Therefore
\[
        AB(J)\ge2\mathsf T(U).
\]
Since $P_N=AA+2AB+BB$ and all three components are nonnegative, we obtain
$P_N(J)\ge4\mathsf T(U)$.
\end{proof}

\begin{lemma}[Density forcing]\label{lem:density-forcing}
Let $k=|L|$.  If $\epsilon=1$, then
\[
 P_N(J)\ge2(2k-H-1)_+(2k-H)_+.
\]
If $\epsilon=2$, then
\[
 P_N(J)\ge2(2k-H-1)_+^2.
\]
Consequently, if $P_N(J)=o(H^2)$ along a sequence, then
\[
        |U|\ge H/2-o(H).
\]
\end{lemma}

\begin{proof}
Let $q_s=\#\{(x,y)\in L^2:x+y=s\}$.  The involutions
$x\mapsto H-x$ on $[H-1]$ and $x\mapsto H+1-x$ on $[H]$ give
\[
        q_H\ge(2k-H-1)_+,
        \qquad q_{H+1}\ge(2k-H)_+.
\]
For $\epsilon=1$, the two sum cells $H$ and $H+1$ are complementary in the
$BB$ correlation.  For $\epsilon=2$, use sums $H$ and $H+2$; the
involution $x\mapsto H+2-x$ on $[2,H]$ gives
$q_{H+2}\ge(2k-H-1)_+$.  The displayed inequalities follow.  Since
$|U|=H-1-k$, the final assertion is immediate.
\end{proof}

\begin{lemma}[Cubic mass of the stable models]\label{lem:stable-model-cubic-mass}
For each $\epsilon\in\{1,2\}$, define the parity and interval upper sets
\[
 U_{\rm par}=\{u\in[H]\setminus\{r\}:u\equiv0\pmod2\},
\]
\[
 U_{\rm int}=[\lfloor(H+1)/2\rfloor,H]\setminus\{r\}.
\]
For $\star\in\{\mathrm{par},\mathrm{int}\}$, put
\[
        L_\star=[H]\setminus\bigl(\{r\}\cup U_\star\bigr),
        \qquad
        J_\star=\{0\}\cup L_\star\cup\{N-u:u\in U_\star\}.
\]
There are absolute constants $c_0>0$ and $H_1$ such that, for all
$H\ge H_1$,
\[
        P_N(J_{\rm par})\ge c_0H^3,
        \qquad P_N(J_{\rm int})\ge c_0H^3,
\]
uniformly in $r$ and $\epsilon$.
\end{lemma}

\begin{proof}
We give uniform lower bounds; the omitted index $r$ changes every displayed
representation count by at most two.

Suppose first that $\epsilon=1$.  In the parity model the positive points
are precisely the odd integers in $[1,2H-1]$, apart from at most one
deletion.  For
\[
        \lceil H/4\rceil\le s\le\lfloor H/2\rfloor,
\]
there are exactly $H-s$ ordered positive-difference pairs of difference
$2s$ before the deletion.  Moreover, writing the odd points as $2i-1$,
the equation
\[
        (2i-1)+(2j-1)=2H+2-2s
\]
has exactly $H+1-s$ ordered solutions in $[H]^2$.  Consequently
\[
        A_{2s}\ge H-s-2,
        \qquad B_{T-2s}\ge H-s-1.
\]
There are $\Omega(H)$ choices of $s$, and both factors are $\Omega(H)$.

Now let $\epsilon=2$.  The lower representatives in the parity model are
the odd integers in $[H]$, again with at most one deletion.  For every even
$q$ with
\[
        \lceil3H/4\rceil\le q\le H,
\]
the number of ordered pairs of these odd integers with sum $q$ is $q/2$
before deletion.  The number with the complementary sum $2H+2-q$ is at
least $q/2-1$.  Hence the two complementary $B$-cells are at least
$q/2-2$ and $q/2-3$, respectively.  Summing their products over the
$\Omega(H)$ admissible even values of $q$ again gives $\Omega(H^3)$.

Finally consider the interval model and put
$a=\lfloor(H+1)/2\rfloor$.  Before the single deletion, its positive points
are
\[
        [1,a-1]\cup[H+\epsilon,2H+\epsilon-a].
\]
For
\[
        \lceil H/4\rceil\le s\le\lfloor H/3\rfloor,
\]
the within-interval positive-difference pairs contribute exactly $H-2s$
before deletion, so
\[
        A_s\ge H-2s-2.
\]
A cross-sum pair can be written as $x+(N-u)$ with
$x\in[1,a-1]$, $u\in[a,H]$, and the equation for the complementary sum is
$u-x=s-1$.  It has $s-1$ choices of $(x,u)$, each in two orders.  Thus
\[
        B_{T-s}\ge2s-4.
\]
Again there are $\Omega(H)$ cells and both factors are $\Omega(H)$,
proving the two uniform cubic lower bounds.
\end{proof}

\begin{lemma}[Lipschitz bound]\label{lem:folded-lipschitz}
There is an absolute constant $C$ such that, for
$J,J'\subseteq[0,N]$ with $|J\triangle J'|\le q$, one has
\[
        |P_N(J)-P_N(J')|\le CqH^2.
\]
\end{lemma}

\begin{proof}
Each of $AA,AB,BA,BB$ counts ordered quadruples satisfying one linear
equation (and $BA=AB$ as above).  For a quadruple containing a prescribed
changed point, choose
its position and two other coordinates; the fourth coordinate is then
determined.  One changed point therefore affects $O(H^2)$ quadruples, and
summing over the changed points proves the claim.
\end{proof}

\begin{theorem}[Uniform folded cover mass]\label{thm:uniform-folded-cover-mass}
There are absolute constants $c>0$ and $H_0$ such that, for every
$H\ge H_0$, every $\epsilon\in\{1,2\}$, every missing index $r$, and every
partition $L\sqcup U=[H]\setminus\{r\}$,
\[
        P_N(J)\ge cH^2.
\]
\end{theorem}
\begin{proof}
If the assertion were false, there would be a sequence with
$P_N(J)=o(H^2)$.  Lemma~\ref{lem:upper-schur-supply} gives
$\mathsf T(U)=o(H^2)$, while Lemma~\ref{lem:density-forcing} gives
$|U|\ge H/2-o(H)$.  Put $n=H+1$ and $V=U+1\subseteq[n]$.  Then
$\mathsf T(U)$ is exactly the number of ordered solutions of $x+y=z$ in
$V$, so Lemma~\ref{lem:schur-stability-consequence} applies.  For each term
of the sequence the set $V$ is $o(H)$-close either to the odd set
$O_{H+1}$ or to the upper half $U_{H+1}$.  In the first case,
\[
 U_{\rm par}
 =\bigl((O_{H+1}-1)\cap[H]\bigr)\setminus\{r\};
\]
the shifted model $O_{H+1}-1$ has only the extraneous point $0$, and
restricting to the missing-index domain can change one further point.  In
the second case,
\[
 U_{H+1}-1=[\lfloor(H+1)/2\rfloor,H],
 \qquad
 U_{\rm int}=(U_{H+1}-1)\setminus\{r\}.
\]
Thus in either alternative one may choose
$U_*\in\{U_{\rm par},U_{\rm int}\}$ with
\[
        |U\triangle U_*|=o(H)+O(1)=o(H).
\]

Let $J_*$ be the corresponding folded model with the same missing index.
Changing one upper index changes at most two points of the folded set, so
$|J\triangle J_*|=O(|U\triangle U_*|)=o(H)$.  By
Lemma~\ref{lem:stable-model-cubic-mass}, $P_N(J_*)\ge c_0H^3$.  By
Lemma~\ref{lem:folded-lipschitz},
\[
        |P_N(J)-P_N(J_*)|=o(H)\,O(H^2)=o(H^3).
\]
Thus $P_N(J)=\Omega(H^3)$, contradicting $P_N(J)=o(H^2)$.

To obtain uniform constants, suppose that none existed.  For each integer
$t$ one could then choose $H_t\ge t$ and an admissible folded set with
$P_N(J_t)<H_t^2/t$.  This is exactly a sequence with
$P_N(J_t)=o(H_t^2)$, contradicting the preceding argument.  Hence suitable
absolute constants $c>0$ and $H_0$ exist.
\end{proof}

\subsection{From folded mass to a surviving template}

Retain the admissible representative set
$\mathcal A=I\setminus\{0\}$ from
Lemma~\ref{lem:folded-representative-structure}.  For $a\in\mathcal A$,
let
\[
        \Pi_a=\{-a,a+1\}\subseteq E.
\]
Define
\[
 \lambda_a(d)=\sum_{i\in C_d}
 \bigl(\1_{-i\in \Pi_a}+\1_{d-i\in \Pi_a}\bigr),
 \qquad
 \Lambda_a(E)=\sum_{d=1}^{m}\lambda_a(d)c_{m+1-d}.
\]

\begin{lemma}[Deletion-degree reduction]\label{lem:deletion-degree}
Every deletable base offset belongs to a unique pair $\Pi_a$ with
$a\in\mathcal A$, and
\[
        \Delta_{\mathrm{del}}(E)\le\max_{a\in\mathcal A}\Lambda_a(E).
\]
\end{lemma}

\begin{proof}
The uniqueness assertion is Lemma~\ref{lem:folded-representative-structure}.
For first-bridge covers, deleting an offset in $\Pi_a$ is counted directly
by $\lambda_a(d)$.  For second-bridge covers, the bijection
$D_d\to C_{m+1-d}$ from Lemma~\ref{lem:profile-folding} turns use of an
offset $\alpha$ into first-bridge use of $1-\alpha$, the other member of
the same reflected pair $\Pi_a$.  After reindexing $d'=m+1-d$, the
first- and second-bridge contributions for $\alpha$ are bounded by the
$\alpha$- and $(1-\alpha)$-parts, respectively, of the single quantity
$\Lambda_a(E)$, rather than by two separate copies of it.  The exclusions
$i=j$ only reduce the deletion degree.  Taking the maximum over $a\in\mathcal A$ proves the
claim.
\end{proof}

For a set $J\subseteq[0,m]$, put
\[
        E_J:=(-J)\cup(J+1)\subseteq\Z_{2m}
\]
and define its folded cover profile by
\begin{equation}\label{eq:folded-cover-profile-definition}
        c_J(d):=\#\{i\in J:d-i\in E_J\},
        \qquad 1\le d\le m.
\end{equation}

For every $a\in\mathcal A$, put
\[
        I_a:=I\setminus\{a\},
        \qquad
        P(I_a):=\sum_{d=1}^{m}c_{I_a}(d)c_{I_a}(m+1-d).
\]

\begin{lemma}[Coarse puncture removal]\label{lem:coarse-puncture-removal}
There is an absolute constant $C$ such that, for every $a\in\mathcal A$,
\[
        P(E)-2\Lambda_a(E)\ge P(I_a)-Cm.
\]
\end{lemma}

\begin{proof}
By Lemma~\ref{lem:folded-representative-structure}, $E=E_I$, and hence
$c_d=c_I(d)$.  Put
\[
        c'_d:=c_{I_a}(d),
        \qquad \mu_d:=c_d-c'_d.
\]
The same lemma shows that deleting $a$ removes
exactly the left index $a$ and exactly the two right offsets
$\Pi_a=\{-a,a+1\}$; neither right offset has another folded
representative.  A lost cover can therefore use the deleted left index
$a$, the deleted right offset $-a$, or the deleted right offset $a+1$.
Hence $0\le\mu_d\le3$ and
\[
        \sum_d\mu_d\le3m.
\]
For $i\in[0,m]$, the condition $-i\in\Pi_a$ is equivalent to
$i=a$: the alternative $-i=a+1$ would give $i=2m-a-1>m$.  Thus a cover is
lost exactly when at least one of the two indicators in
$\lambda_a(d)$ is nonzero.  The quantity $\mu_d$ counts their logical OR
over all covers, whereas $\lambda_a(d)$ counts their sum.  Their
difference is therefore the number of covers for which both indicators are
one.  Such a cover must have $i=a$ and $d-a=a+1$, so it can occur only in
the single profile cell $d=2a+1\le m$.  Hence
\[
        \lambda_a(d)=\mu_d+\varepsilon_d,
        \qquad \varepsilon_d\in\{0,1\},
        \qquad |\operatorname{supp}(\varepsilon)|\le1.
\]
By symmetry under $d\mapsto m+1-d$,
\[
\begin{aligned}
 P(I_a)
 &=P(E)-2\sum_d\mu_d\,c_{m+1-d}
   +\sum_d\mu_d\mu_{m+1-d},\\
 \Lambda_a(E)
 &=\sum_d(\mu_d+\varepsilon_d)c_{m+1-d}.
\end{aligned}
\]
Therefore
\[
 P(E)-2\Lambda_a(E)
 =P(I_a)-\sum_d\mu_d\mu_{m+1-d}
       -2\sum_d\varepsilon_d\,c_{m+1-d}.
\]
The first error is at most $3\sum_d\mu_d\le9m$.  Also
$C_d\subseteq I$, so
\[
        c_d\le |I|\le \frac m2
\]
for both reflected-wave types.  Since
$|\operatorname{supp}(\varepsilon)|\le1$, the second error is at most $m$.
This proves the lemma with, for example, $C=10$.
\end{proof}

\begin{lemma}[Parity-uniform folding after one deletion]
\label{lem:parity-uniform-folding}
For every $a\in\mathcal A$ and every $1\le d\le m$,
\[
        c_{I_a}(d)=A_d(I_a)+B_d(I_a).
\]
Moreover $I_a$ has the form used in
Theorem~\ref{thm:uniform-folded-cover-mass}, with
\[
(H,N,\epsilon)=
\begin{cases}
\bigl((m-2)/2,m-1,1\bigr),&m\text{ even},\\
\bigl((m-3)/2,m-1,2\bigr),&m\text{ odd},
\end{cases}
\]
and
\[
        P(I_a)=P_N(I_a).
\]
\end{lemma}

\begin{proof}
Put $N=m-1$ and recall
\[
        I=\{i\in[0,m]:-i\in E\}.
\]
We first identify the reflected representative pairs in $I$.  Let $H$ have
the value stated in the lemma and fix $u\in[H]$.  Then
\[
 u\in I\quad\Longleftrightarrow\quad -u\in E,
\]
whereas
\[
 N-u\in I
 \quad\Longleftrightarrow\quad
 u+1-m\in E.
\]
The offset $u+1-m$ is the antipode of $u+1$, and reflection gives
\[
        u+1\in E\quad\Longleftrightarrow\quad -u\in E.
\]
For a full wave, antipodal transversality therefore says that exactly one
of $u$ and $N-u$ belongs to $I$.  The identical argument applies to the
odd zero-free wave for every $u\in[H]$, because the omitted central
antipodal pair is not one of these pairs.  Also $0\in I$, while
$m,N\notin I$: the offsets $m$ and $m+1$ are antipodes of the anchors
$0$ and $1$.  When $m=2h+1$ is odd, the only remaining index is
$h=(m-1)/2$, and $h\notin I$ because $-h$ is a member of the omitted
central antipodal pair $\{-h,h+1\}$.

It follows that
\[
        I=\{0\}\cup\{\text{one element of }\{u,N-u\}:u\in[H]\}.
\]
Deleting a nonanchor representative $a$ removes the representative of one
unique pair, say the pair indexed by $r$.  Thus there is a unique partition
\[
        L\sqcup U=[H]\setminus\{r\}
\]
for which
\[
        I_a=\{0\}\cup L\cup\{N-u:u\in U\}.
\]
This is precisely the folded form used in
Theorem~\ref{thm:uniform-folded-cover-mass}; moreover
$N=2H+1$ when $m$ is even and $N=2H+2$ when $m$ is odd.

It remains to identify the profile.  By
\eqref{eq:folded-cover-profile-definition}, a cover indexed by
$i\in I_a$ satisfies either $d-i\equiv-j$ or
$d-i\equiv j+1\pmod{2m}$ for some $j\in I_a$.  These congruences are
ordinary integer equations in the present ranges.  The only possible
wrapping endpoint in the first congruence would require $j=m$, but
$m\notin I_a$; the second congruence cannot wrap.  Hence the two equations
are exactly
\[
        i-j=d,
        \qquad i+j=d-1.
\]
By Lemma~\ref{lem:folded-representative-structure},
\[
        E_{I_a}=(-I_a)\mathbin{\dot\cup}(I_a+1),
\]
so these two alternatives are mutually exclusive.  Therefore
\[
        c_{I_a}(d)=A_d(I_a)+B_d(I_a).
\]
Finally the complementary profile constant is
$m+1=N+2$, and the summation ranges agree because $N+1=m$.  Therefore
$P(I_a)=P_N(I_a)$.
\end{proof}

\begin{proposition}[Reflected-wave cover surplus]\label{prop:reflected-wave-cover-surplus}
There are constants $c>0$ and $m_0$ such that every reflected base wave
$E$ arising in Corollary~\ref{cor:one-extra-deletion-normal-form} satisfies
\[
        Q(E)-\Delta_{\mathrm{del}}(E)\ge cm^2
\]
for $m\ge m_0$.
\end{proposition}

\begin{proof}
Let $c>0$ be the constant from
Theorem~\ref{thm:uniform-folded-cover-mass}.  In both parity cases of
Lemma~\ref{lem:parity-uniform-folding}, one has $H\ge m/4$ for all
sufficiently large $m$.  Hence there is a new absolute constant
$c_2>0$ such that
\[
        P(I_a)=P_N(I_a)\ge c_2m^2
\]
uniformly in $a$; for example, one may take $c_2=c/16$.  By
Lemma~\ref{lem:coarse-puncture-removal}, after increasing the threshold
there is $c_3>0$ such that
\[
        P(E)-2\Lambda_a(E)\ge c_3m^2
\]
for every $a\in\mathcal A$.  For sufficiently large $m$ the set
$\mathcal A$ is nonempty.  Hence, using
Lemmas~\ref{lem:profile-folding} and~\ref{lem:deletion-degree},
\[
\begin{aligned}
 Q(E)-\Delta_{\mathrm{del}}(E)
 &\ge \frac12P(E)-\max_{a\in\mathcal A}\Lambda_a(E)\\
 &=\frac12\min_{a\in\mathcal A}
       \bigl(P(E)-2\Lambda_a(E)\bigr)
 \ge \frac{c_3}{2}m^2.
\end{aligned}
\]
\end{proof}

\begin{theorem}[The central middle length]\label{thm:central-middle-length}
For all sufficiently large $m$, every even one-extra-edge setup contains a
Berge cycle of length $m+1$.
\end{theorem}
\begin{proof}
As observed at the beginning of Section~\ref{sec:central-length}, a
counterexample is central and antipodal-free.  Corollary~\ref{cor:one-extra-deletion-normal-form}
then supplies a reflected base wave with at most one deleted base offset in
each row and at most $2m$ deletions in total.  By
Proposition~\ref{prop:reflected-wave-cover-surplus}, $Q(E)>\Delta_{\mathrm{del}}(E)$ for sufficiently
large $m$.  Lemma~\ref{lem:global-candidate-survival} gives a surviving two-gap
cycle of length $m+1$.
\end{proof}

\section{The remaining middle lengths}\label{sec:six-middle}

Throughout this section assume Setup~\ref{setup:even-outside}; in
particular, $m\ge5$, the Hamiltonian labels are $e_p$, and
$f\notin E(C)$ is an additional $(m-1)$-edge.  It remains to obtain
\[
        m-2,\quad m-1,\quad m,\quad m+2,\quad m+3,\quad m+4.
\]
These lengths occur in three chord-dual pairs.  Fix
$j\in\{1,2,3\}$ and put $s=m-j$.  If $s\in\Delta(f)$, the chord lemma gives
both $s+1=m-j+1$ and $2m-s+1=m+j+1$.  Hence, for this fixed value of $j$,
we may assume $s\notin\Delta(f)$.  Thus $A:=f\subseteq\Z_{2m}$ has size
$m-1$ and satisfies
\[
        A\cap(A+s)=\varnothing.
\]
Write
\[
 \mathfrak D=\mathfrak D_s(A)
 :=\Z_{2m}\setminus\bigl(A\cup(A+s)\bigr),
 \qquad
 \alpha_x=\mathbf 1_{x\in A},
 \qquad
 \delta_x=\mathbf 1_{x\in\mathfrak D}.
\]

\begin{lemma}[Exact defect recurrence]
\label{lem:defect-recurrence}
One has $|\mathfrak D|=2$ and, for every $x\in\Z_{2m}$,
\begin{align}
 \alpha_x+\alpha_{x-s}&=1-\delta_x, \label{eq:defect-first-recurrence}\\
 \alpha_{x+2j}-\alpha_x&=\delta_x-\delta_{x-s}\label{eq:defect-second-recurrence}.
\end{align}
In particular, translation by $2j$ changes the membership word of $A$ on
at most four directed edges in total.
\end{lemma}

\begin{proof}
The disjoint sets $A$ and $A+s$ both have size $m-1$, so their union has
size $2m-2$ and $|\mathfrak D|=2$.  At a vertex $x$, membership in $A+s$ is
$\alpha_{x-s}$, which proves \eqref{eq:defect-first-recurrence}.  Subtract
that equation from its translate by $-s$ and use
$-2s\equiv2j\pmod{2m}$ to obtain \eqref{eq:defect-second-recurrence}.
\end{proof}

\subsection{All-split exclusions at a defect row}
\label{subsec:all-split}

For a row $p$, retain
\[
 G_p=A\cap[p+s+1,p]^{+},
 \qquad
 H_p=A\cap[p+1,p-s]^{+}.
\]
Define
\phantomsection
\begin{equation*}\tag{L}\label{eq:all-split-long}
\begin{split}
 \mathcal F_p^{\rm long}:={}&
 \{p-(s-1),p+s\}\\
 &\cup\!\bigcup_{\substack{1\le D\le s-1\\p+D\in A}}
          (G_p+D-s-1)\\
 &\cup\!\bigcup_{\substack{1\le R\le s-1\\p-R\in A}}
          (H_p+s-R).
\end{split}
\end{equation*}
and
\phantomsection
\begin{equation*}\tag{S}\label{eq:all-split-short}
\begin{split}
 \mathcal F_p^{\rm short}:={}&
 \{p+s+1,p-s\}\\
 &\cup\!\bigcup_{\substack{1\le D\le s-1\\p+D\in A}}
          (G_p+D-s)\\
 &\cup\!\bigcup_{\substack{1\le R\le s-1\\p-R\in A}}
          (H_p+s-R-1).
\end{split}
\end{equation*}
The long target is $2m-s+1=m+j+1$ and the short target is
$s+1=m-j+1$.

\begin{lemma}[All-split exclusion]
\label{lem:all-split-exclusion}
If the long target is absent, then
\[
        e_p\cap\mathcal F_p^{\rm long}=\varnothing.
\]
If the short target is absent, then
\[
        e_p\cap\mathcal F_p^{\rm short}=\varnothing.
\]
Consequently either target occurs whenever the corresponding forbidden set
has size at least $m+2$ for some $p$.
\end{lemma}

\begin{proof}
The singleton terms are the two boundary-bypass exclusions.  The remaining
terms are exactly all row-first and $f$-first split locks with
$D+R=s$.  Thus absence of the target forces $e_p$ to avoid their union.  A
forbidden set of size at least $m+2$ leaves at most $m-2$ available vertices,
contrary to $|e_p|=m-1$.
\end{proof}

The next identity is the algebraic core of the proof.  It converts all split
locks into comparisons on two adjacent additive diagonals.

\begin{lemma}[Defect-row word form]
\label{lem:defect-row-word-form}
Let $p\in\mathfrak D$ and translate so that $p=0$.  For
$1\le D,t\le s-1$, the short forbidden set contains
\begin{align}
 &D+t
 &&\text{if }\alpha_D=1,\ \alpha_t=0,\ \delta_{s+t}=0,
 \label{eq:defect-short-first}\\*
 &D+t-1
 &&\text{if }\alpha_D=0,\ \alpha_t=1,\ \delta_D=0,
 \label{eq:defect-short-second}
\end{align}
and the long forbidden set contains
\begin{align}
 &D+t-1
 &&\text{if }\alpha_D=1,\ \alpha_t=0,\ \delta_{s+t}=0,
 \label{eq:defect-long-first}\\*
 &D+t
 &&\text{if }\alpha_D=0,\ \alpha_t=1,\ \delta_D=0.
 \label{eq:defect-long-second}
\end{align}
All four displayed values are ordinary integers in $[1,2m-2]$, so no modular
wrap is involved.
\end{lemma}

\begin{proof}
Since $0\in\mathfrak D$, both $0$ and $-s$ are absent from $A$.  Write an element of
$G_0$ as $s+t$, where $1\le t\le2m-s-1$.  By
\eqref{eq:defect-first-recurrence},
\[
        \alpha_{s+t}=1-\delta_{s+t}-\alpha_t.
\]
Thus $\alpha_D=1$, $\alpha_t=0$, and $\delta_{s+t}=0$ make the row-first
lock applicable.  Its short and long translates are respectively
\[
 (s+t)+D-s=D+t,
 \qquad
 (s+t)+D-s-1=D+t-1.
\]
Similarly,
\[
        \alpha_{D-s}=1-\delta_D-\alpha_D.
\]
If $\alpha_D=0$, $\alpha_t=1$, and $\delta_D=0$, the $f$-first lock is
applicable with $R=s-D$ and $t\in H_0$.  Its short and long translates are
$D+t-1$ and $D+t$.  Restricting to $t\le s-1$ only discards valid witnesses.
\end{proof}

\subsection{Short monochromatic runs}
\label{subsec:run-bounds}

A monochromatic run means a cyclic interval on which the word
$(\alpha_x)$ is constant, with value either zero or one.

\begin{lemma}[Run bounds]
\label{lem:run-bounds}
Suppose $A\cap(A+m-j)=\varnothing$ and $|A|=m-1$.
\begin{enumerate}[label=\textup{(\roman*)}]
\item If $j=2$, every monochromatic run has length at most $3$.
\item If $j=3$, every monochromatic run has length at most $4$.
\end{enumerate}
\end{lemma}

\begin{proof}
We use \eqref{eq:defect-second-recurrence}.  Put
$g=\gcd(2m,2j)$.  Along the $g$ cycles generated by translation by $2j$,
there are at most four transition edges altogether.  Every nonconstant
cycle has at least two transitions, so at most two of these cycles are
nonconstant.

We need a paired-cycle refinement only in the two divisible cases below.
If $j=2$ and $4\mid m$, then $g=4$ and $s\equiv2\pmod4$; if
$j=3$ and $6\mid m$, then $g=6$ and $s\equiv3\pmod6$.  In either case a
defect in residue class $c$ creates possible transition edges in the two
classes $c$ and $c+s$.  A cyclic word cannot have exactly one transition.
Since there are only two defects, all nonconstant residue cycles therefore
belong to at most one pair $\{c,c+s\}$.

In fact, in each of these two divisible cases the two defects lie in the
same residue pair.  Otherwise every residue cycle would contain at most one
possible transition edge in
\eqref{eq:defect-second-recurrence}.  A nonconstant cyclic binary word has at
least two transitions, so every residue cycle would be constant.  The
recurrence would then give
\[
        \delta_x=\delta_{x-s}\qquad\text{for every }x,
\]
and hence $\mathfrak D=\mathfrak D+s$.  This is impossible: when
$j=2$ and $4\mid m$, translation by $s=m-2$ has orbit length $m>2$;
when $j=3$ and $6\mid m$, translation by $s=m-3$ has orbit length
$2m/3>2$.  Thus both defects lie in one residue pair.  Every other residue
pair is defect-free, and
\eqref{eq:defect-first-recurrence} shows that its two constant values are
opposite.

First take $j=2$.  If $m\equiv2\pmod4$, the step-$s$ graph consists of four
odd cycles and has independence number $m-2$, so no such $A$ exists.  If
$m\equiv0\pmod4$, then $g=4$ and $s\equiv2\pmod4$.  At most one of the
residue pairs
\[
        \{0,2\},\qquad \{1,3\}
\]
is nonconstant.  The other pair is constant, and its two constants are
opposite by \eqref{eq:defect-first-recurrence}.  Four consecutive integers
contain both members of that constant pair, so they cannot be monochromatic.

It remains to take $j=2$ and $m$ odd.  Then $\gcd(2m,s)=1$, so the step-$s$
graph is one even cycle.  In its cyclic coordinate, an independent set of
size $m-1$ has $m-1$ ones and $m+1$ zeros.  The $m-1$ nonempty zero-gaps
between successive ones therefore contain exactly two excess zeros, so the
cycle has exactly two zero--zero edges.  Put $b_t=\alpha_{ts}$ and
\[
        \varphi_t=b_t\oplus(t\bmod2).
\]
Because adjacent ones are forbidden, one has
$\varphi_t\ne\varphi_{t+1}$ exactly when $b_t=b_{t+1}=0$.  Thus
$\varphi$ has exactly two transitions.

Let $u=s^{-1}\pmod{2m}$.  It is odd and satisfies
\[
        2u\equiv m-1\pmod{2m}.
\]
If $m\equiv3\pmod4$, then $u=(m-1)/2$, and the four step coordinates of
four ordinary consecutive vertices occur cyclically as
\[
        0,\,u,\,m-1,\,u+m-1.
\]
If $m\equiv1\pmod4$, then $u=(3m-1)/2$ and
$u+m-1\equiv(m-3)/2\pmod{2m}$; their cyclic order is
\[
        0,\,u+m-1,\,m-1,\,u.
\]
In the first order the corresponding ordinary-coordinate indices are
$0,1,2,3$, and in the second they are $0,3,2,1$.  Since $u$ is odd and
$m-1$ is even, equal ordinary membership bits require alternating values
of $\varphi$ at the four cyclic points in either case.  This would force at
least four transitions, contradicting the two-transition form of
$\varphi$.  The same argument applies to a zero run.

Now take $j=3$.  If $m\equiv3\pmod6$, the step-$s$ graph consists of six
odd cycles and has independence number $m-3$, so again no such $A$ exists.
If $m\equiv0\pmod6$, then $g=6$ and $s\equiv3\pmod6$.  At most one of the
three residue pairs
\[
        \{0,3\},\qquad\{1,4\},\qquad\{2,5\}
\]
is nonconstant; every other pair is constant with opposite values.  Five
consecutive residues contain both members of two of these three pairs, and
therefore contain both values.  Hence no monochromatic run has length five.

Suppose next that $m$ is even and $3\nmid m$.  Then $\gcd(2m,s)=1$, and the
same two-transition phase word $\varphi$ is available.  Let
$u=s^{-1}\pmod{2m}$.  If $m\equiv4\pmod6$, then
\[
        u=\frac{m-1}{3}.
\]
Here
\[
        0<u<2u<3u=m-1<4u<2m,
\]
so the five points $0,u,2u,3u,4u$ occur in that cyclic order.  Their
ordinary-coordinate indices are $0,1,2,3,4$, with parity sequence
even, odd, even, odd, even and four cyclic changes.  If $m\equiv2\pmod6$, then
\[
        u=\frac{5m-1}{3},
\]
and reduction modulo $2m$ gives
\[
 4u\equiv\frac{2m-4}{3},\qquad
 3u\equiv m-1,\qquad
 2u\equiv\frac{4m-2}{3}.
\]
These residues satisfy
\[
 0<4u<3u<2u<u<2m,
\]
so the cyclic index order is $0,4,3,2,1$, whose parity sequence is
even, even, odd, even, odd, again with four cyclic changes.  Five equal ordinary bits
would therefore force at least four transitions of $\varphi$, a
contradiction.

Finally suppose that $m$ is odd and $3\nmid m$.  The step-$s$ graph consists
of two odd cycles, one on each parity class, and $A$ is maximum independent
on both.  A run of five equal ordinary bits contains three bits of one
parity, corresponding in the parity quotient $\Z_m$ to three
consecutive ordinary points.  Put $a=(m-3)/2$ and $u=a^{-1}\pmod m$.  If
$m\equiv1\pmod6$, then $u=2(m-1)/3$ and the three step-cycle arcs have
lengths
\[
        \frac{m-4}{3},\qquad
        \frac{m+2}{3},\qquad
        \frac{m+2}{3}.
\]
If $m\equiv5\pmod6$, then $u=(m-2)/3$ and the arc lengths are
\[
        \frac{m-2}{3},\qquad
        \frac{m-2}{3},\qquad
        \frac{m+4}{3}.
\]
All three lengths are odd.  In a maximum independent set of an odd cycle,
the cyclic gaps between consecutive selected vertices are all $2$ except
for one gap of length $3$; between consecutive unselected vertices they
are all $2$ except for one gap of length $1$.  Each of the three arcs determined by three vertices of the same
membership value is a union of consecutive same-value gaps, and the unique
odd gap belongs to exactly one of those arcs.  Hence at most one of the
three arcs is odd.  The three displayed odd arcs are impossible.  This proves the
$j=3$ bound for both one-runs and zero-runs.
\end{proof}

\subsection{The reflection argument on adjacent diagonals}
\label{subsec:reflection-diagonals}

\begin{lemma}[A missing lock value forces a long constant block]
\label{lem:reflection-block}
Assume every monochromatic run of $A$ has length at most $R$, and let
$p\in\mathfrak D$.  For
$\star\in\{\mathrm{short},\mathrm{long}\}$, let
$\mathcal F_p^\star$ denote the corresponding forbidden set from
\eqref{eq:all-split-short} or \eqref{eq:all-split-long}.  Then
\begin{equation}\label{eq:reflection-complement-size}
 \left|\Z_{2m}\setminus\mathcal F_p^\star\right|
 \le 10R+2j+2.
\end{equation}
More precisely, after translating $p$ to zero and using representatives in
$[0,2m-1]$,
\begin{equation}\label{eq:reflection-boundary-containment}
 \Z_{2m}\setminus\mathcal F_p^\star
 \subseteq
 [0,5R]\ \cup\ [2s-5R-1,2m-1].
\end{equation}
\end{lemma}

\begin{proof}
Fix
$\star\in\{\mathrm{short},\mathrm{long}\}$.  Fix $x$ with
$2\le x\le2s-3$ and define
\[
 I_x=
 \bigl[\max\{1,x-s+1\},\ \min\{s-2,x-1\}\bigr].
\]
For every $i\in I_x$, the three indices
\[
        i,\qquad x-i,\qquad i+1
\]
lie in $[1,s-1]$.

If $\star=\mathrm{short}$ and
$x\notin\mathcal F_0^{\rm short}$, apply
\eqref{eq:defect-short-first} to the two ordered decompositions
$x=i+(x-i)=(x-i)+i$.  Unless one of $s+i$ and $s+x-i$ is the other defect,
the two resulting inequalities give
\[
        \alpha_i=\alpha_{x-i}.
\]
Apply \eqref{eq:defect-short-second} to
$x+1=(x-i)+(i+1)=(i+1)+(x-i)$.  Unless one of $x-i$ and $i+1$ is the other
defect, it gives
\[
        \alpha_{x-i}=\alpha_{i+1}.
\]
Thus $\alpha_i=\alpha_{i+1}$ for every $i\in I_x$ apart from at most four
values of $i$.

If $\star=\mathrm{long}$ and
$x\notin\mathcal F_0^{\rm long}$, use
\eqref{eq:defect-long-second} on the diagonal of sum $x$ and
\eqref{eq:defect-long-first} on the diagonal of sum $x+1$.  The same
conclusion follows, again with at most four exceptional values.

The word on the $|I_x|+1$ consecutive positions from the first element of
$I_x$ through one past its last element therefore has at most four changes.
It is a union of at most five monochromatic runs, each of length at most
$R$.  Hence
\[
        |I_x|+1\le5R.
\]
But
\[
        |I_x|=\min\{x-1,\,2s-x-2\}.
\]
Therefore $x\le5R$ or $x\ge2s-5R-1$.  The representatives outside
$[2,2s-3]$ already lie in one of those two boundary intervals.  This proves
\eqref{eq:reflection-boundary-containment}.  Its right side has at most
\[
 (5R+1)+(2m-2s+5R+1)=10R+2j+2
\]
points, proving \eqref{eq:reflection-complement-size}.
\end{proof}

\begin{theorem}[Uniform defect-row forcing for $j=2,3$]
\label{thm:defect-row-forcing}
Let $A\subseteq\Z_{2m}$ have size $m-1$ and miss $s=m-j$.
\begin{enumerate}[label=\textup{(\roman*)}]
\item If $j=2$ and $m\ge38$, then every defect row $p\in\mathfrak D_s(A)$ satisfies
\[
 |\mathcal F_p^{\rm short}|\ge m+2,
 \qquad
 |\mathcal F_p^{\rm long}|\ge m+2.
\]
\item If $j=3$ and $m\ge50$, the same conclusion holds.
\end{enumerate}
\end{theorem}
\begin{proof}
For $j=2$, Lemmas~\ref{lem:run-bounds} and~\ref{lem:reflection-block} give a complement of size at most
\[
        10\cdot3+2\cdot2+2=36\le m-2.
\]
For $j=3$, the corresponding bound is
\[
        10\cdot4+2\cdot3+2=48\le m-2.
\]
In either case the forbidden set has at least $2m-(m-2)=m+2$ points.
\end{proof}

\begin{corollary}[The four $j=2,3$ middle lengths]
\label{cor:j23-middle-lengths}
For every $m\ge50$, if the extra edge misses $m-j$ with
$j\in\{2,3\}$, then both associated lengths
\[
        m-j+1,
        \qquad
        m+j+1
\]
occur.
\end{corollary}

\begin{proof}
Combine Theorem~\ref{thm:defect-row-forcing} and
Lemma~\ref{lem:all-split-exclusion}.
\end{proof}

\subsection{The distance-\texorpdfstring{$(m-1)$}{(m-1)} case}
\label{subsec:j1}

Here $s=m-1$.  The two-defect recurrence can be solved explicitly.
A dihedral change of coordinates is always applied to the entire labeled
Hamiltonian setup.  Under a rotation $\phi(x)=x+t$, define
$e'_{p+t}=\phi(e_p)$.  Under a reflection $\phi(x)=c-x$, define
\[
        e'_{c-p-1}=\phi(e_p),
\]
because the adjacent pair indexed by $p$ is carried to the adjacent pair
indexed by $c-p-1$.  In both cases replace $f$ by $\phi(f)$.  Thus all
incidences, edge distinctness, and the row labels used below are preserved.

\begin{lemma}[Normal form for a distance-$(m-1)$-free edge]
\label{lem:j1-normal-form}
After a dihedral change of coordinates, every $(m-1)$-set
$A\subseteq\Z_{2m}$ satisfying
$A\cap(A+m-1)=\varnothing$ has the following form.  For some
$0\le a\le\lfloor(m-1)/2\rfloor$, put $\delta=2a+1$; then
\begin{equation}\label{eq:j1-normal-form}
 A=A_a:=
 [2,2a]_2
 \cup[2a+2,m]
 \cup[m+2,m+2a]_2.
\end{equation}
Its two defects are $0$ and $\delta$.  The case $a=0$ is a cyclic interval
of length $m-1$.
\end{lemma}

\begin{proof}
The step-$(m-1)$ graph is one even cycle when $m$ is even and two odd
cycles when $m$ is odd.  An independent set of size $m-1$ has exactly two
zero--zero defect edges in total.  Rotate one defect to $0$ and reflect so
that the other defect $\delta$ lies in $[1,m]$.

For $j=1$, recurrence \eqref{eq:defect-second-recurrence} becomes
\begin{equation}\label{eq:j1-step-two-recurrence}
        \alpha_{x+2}-\alpha_x=\delta_x-\delta_{x-(m-1)}.
\end{equation}
Thus the possible nonzero increments occur at
\[
        0,\quad \delta,\quad m-1,\quad \delta+m-1,
\]
with coincident contributions combined.  If $m$ is odd, the step-$(m-1)$
graph has one defect in each parity class, so $0$ and $\delta$ have
opposite parity.  If $m$ is even and $\delta\ne m-1$, the four displayed
positions are distinct.  Summing \eqref{eq:j1-step-two-recurrence} around
each of the two step-$2$ cycles shows that each parity class has total
increment zero.  If $\delta$ were even, both positive increments would lie
on the even class and both negative increments on the odd class, which is
impossible.  Hence $\delta$ is odd in this case as well.  The remaining
even boundary case is $\delta=m-1$, which is also odd.  Therefore
\[
        \delta=2a+1
\]
for some $0\le a\le\lfloor(m-1)/2\rfloor$.

We now solve the two parity subsequences.  Since $0$ and $\delta$ are
defects, \eqref{eq:defect-first-recurrence} at $x=0$ and $x=\delta$ gives
\[
        \alpha_0=\alpha_{m+1}=0,
        \qquad
        \alpha_\delta=\alpha_{\delta-s}=0.
\]

Suppose first that $m$ is odd.  On the even subsequence the $+1$ increment
is at $0$ and the $-1$ increment is at $m-1$; hence its ones are
$2,4,\ldots,m-1$.  The value $\alpha_\delta=0$ supplies the initial value on the odd
subsequence.  Its $+1$
increment is at $\delta$ and its $-1$ increment is at
$\delta+m-1$; hence its ones are
\[
        \delta+2,\delta+4,\ldots,\delta+m-1.
\]
Their union is exactly the set in \eqref{eq:j1-normal-form}.

Now suppose that $m$ is even and $\delta<m-1$.  On the even subsequence the
$+1$ increment is at $0$ and the $-1$ increment is at
$\delta+m-1=m+2a$; its ones are therefore
$2,4,\ldots,m+2a$.  On the odd subsequence the $+1$ increment is at
$\delta$ and the $-1$ increment is at $m-1$; its ones are
$\delta+2,\delta+4,\ldots,m-1$.  Again their union is precisely
\eqref{eq:j1-normal-form}.

Finally let $m$ be even and $\delta=m-1$.  At $x=m-1$ the positive and
negative contributions in \eqref{eq:j1-step-two-recurrence} cancel.  The
only nonzero increments are $+1$ at $0$ and $-1$ at $2m-2$.  Hence the
even subsequence consists of all nonzero even residues, while the odd
subsequence is constant.  The even subsequence already contains $m-1$
points, so the size condition forces the odd subsequence to be identically
zero.  This is \eqref{eq:j1-normal-form} with
$a=(m-2)/2$.

Conversely, the displayed set has size $m-1$, its only zero--zero
step-$(m-1)$ edges are the defects $0$ and $2a+1$, and it contains no
step-$(m-1)$ edge.  Recurrence \eqref{eq:defect-first-recurrence} determines
all membership bits from the two defect positions, so the list is complete.
\end{proof}

\begin{lemma}[Noninterval $j=1$ states are forced]
\label{lem:j1-noninterval}
Let $m\ge8$ and $a\ge1$ in Lemma~\ref{lem:j1-normal-form}.  Put
$p=2a=\delta-1$.  Then both forbidden sets have complements of size at
most six.  More precisely:
\begin{align}
 a=1:\quad
 &\Z_{2m}\setminus\mathcal F_p^{\rm short}
   =\{2,3,4,m+4\},\notag\\
 &\Z_{2m}\setminus\mathcal F_p^{\rm long}
   =\{1,2,3,m+2,m+3\};\notag\\[1mm]
 2\le a,\ 2a+1<m:\quad
 &\Z_{2m}\setminus\mathcal F_p^{\rm short}
   =\{2a-2,2a,2a+1,2a+2\},\notag\\
 &\Z_{2m}\setminus\mathcal F_p^{\rm long}
   =\{2a-1,2a,2a+1,m+1\};\notag\\[1mm]
 2a+1=m:\quad
 &\Z_{2m}\setminus\mathcal F_p^{\rm short}\notag\\
 &\quad=\{m-3,m-1,m\}\notag\\
 &\qquad{}\cup\{m+1,m+2,m+4\},\notag\\
 &\Z_{2m}\setminus\mathcal F_p^{\rm long}
   =\{m-2,m-1,m,m+1,m+3\}.
 \label{eq:j1-complement-identities}
\end{align}
Consequently both lengths $m$ and $m+2$ occur.
\end{lemma}

\begin{proof}
Put $s=m-1$ and $p=2a$.  Using the step-two interval notation from
Section~\ref{sec:prelim}, put
\[
 \mathcal D_a:=\{D\in[1,m-2]:p+D\in A_a\},
 \qquad
 \mathcal R_a:=\{R\in[1,m-2]:p-R\in A_a\}.
\]
The normal form \eqref{eq:j1-normal-form} and the definitions of
$G_p,H_p$ give the following complete list:
\begin{align}
 a=1:\quad
 &G_p=\{2,m+2\},\notag\\[-1mm]
 &H_p=[4,m]\cup\{m+2\};\notag\\[1mm]
 2\le a,\ 2a+1<m:\quad
 &G_p=[2,2a]_2\cup\{m+2a\},\notag\\[-1mm]
 &H_p=[2a+2,m]\cup[m+2,m+2a]_2;\notag\\[1mm]
 2a+1=m:\quad
 &G_p=[2,m-1]_2\cup\{2m-1\},\notag\\[-1mm]
 &H_p=[m+2,2m-1]_2.
 \label{eq:j1-GH-table}
\end{align}
and
\begin{align}
 a=1:\quad
 &\mathcal D_a=[2,m-2],\qquad \mathcal R_a=\varnothing;\notag\\[1mm]
 2\le a,\ 2a+1<m:\quad
 &\mathcal D_a=[2,m-2a]\cup[m-2a+2,m-2]_2,\notag\\[-1mm]
 &\mathcal R_a=[2,2a-2]_2;\notag\\[1mm]
 2a+1=m:\quad
 &\mathcal D_a=[3,m-2]_2,\qquad
 \mathcal R_a=[2,m-3]_2.
 \label{eq:j1-DR-table}
\end{align}
Equations~\eqref{eq:j1-GH-table} and~\eqref{eq:j1-DR-table} are
obtained by intersecting the three pieces of $A_a$ with the two cyclic
intervals in \eqref{eq:split-GH-definitions}; no cycle construction is
being suppressed.

Since $s=m-1$, separate the row-first and $f$-first pieces as
\begin{align*}
 \mathcal U_{D}^{\rm short}
 &:=\bigcup_{D\in\mathcal D_a}(G_p+D-m+1),&
 \mathcal U_{R}^{\rm short}
 &:=\bigcup_{R\in\mathcal R_a}(H_p+m-R-2),\\
 \mathcal U_{D}^{\rm long}
 &:=\bigcup_{D\in\mathcal D_a}(G_p+D-m),&
 \mathcal U_{R}^{\rm long}
 &:=\bigcup_{R\in\mathcal R_a}(H_p+m-R-1).
\end{align*}
Then
\begin{align*}
 \mathcal F_p^{\rm short}
 &=\{p+m,p-m+1\}
   \cup\mathcal U_D^{\rm short}\cup\mathcal U_R^{\rm short},\\
 \mathcal F_p^{\rm long}
 &=\{p-m+2,p+m-1\}
   \cup\mathcal U_D^{\rm long}\cup\mathcal U_R^{\rm long}.
\end{align*}
All sets below are written in the representatives $0,1,\ldots,2m-1$.
Substituting \eqref{eq:j1-GH-table} and \eqref{eq:j1-DR-table}, taking the
four elementary sumsets in each line, and reducing modulo $2m$ gives the
following complete expansions.

If $a=1$, then
\begin{align*}
 \mathcal U_D^{\rm short}
  &=[0,1]\cup[5,m+1]\cup[m+5,2m-1],&
 \mathcal U_R^{\rm short}&=\varnothing,\\
 \mathcal U_D^{\rm long}
  &=\{0\}\cup[4,m]\cup[m+4,2m-1],&
 \mathcal U_R^{\rm long}&=\varnothing.
\end{align*}

Suppose that $2\le a$ and $2a+2<m$.  Then
\begin{align*}
 \mathcal U_D^{\rm short}
  ={}&[0,1]\cup[3,2a-1]_2\cup[2a+3,m+1]\\
    &{}\cup\{m+3\}\cup[m+5,2m-1],\\
 \mathcal U_R^{\rm short}
  ={}&[0,2a-4]_2\cup[m+2,2m-4]\cup\{2m-2\},\\
 \mathcal U_D^{\rm long}
  ={}&[0,2a-2]_2\cup[2a+2,m]\cup\{m+2\}\\
    &{}\cup[m+4,2m-1],\\
 \mathcal U_R^{\rm long}
  ={}&[1,2a-3]_2\cup[m+3,2m-3]\cup\{2m-1\}.
\end{align*}

The remaining even endpoint of the same complement case is $m=2a+2$.
Here the ordinary interval part of $\mathcal D_a$ is a singleton, and the
expansions are
\begin{align*}
 \mathcal U_D^{\rm short}
  &=[1,m-3]_2\cup[m+1,2m-1]_2,\\
 \mathcal U_R^{\rm short}
  &=[0,m-6]_2\cup[m+2,2m-2]_2,\\
 \mathcal U_D^{\rm long}
  &=[0,m-4]_2\cup[m,2m-2]_2,\\
 \mathcal U_R^{\rm long}
  &=[1,m-5]_2\cup[m+3,2m-1]_2.
\end{align*}

Finally, if $m=2a+1$, then
\begin{align*}
 \mathcal U_D^{\rm short}
  &=[1,m-2]_2\cup\{m+3\}\cup[m+5,2m-1],\\
 \mathcal U_R^{\rm short}
  &=[0,m-5]_2\cup[m+3,2m-2]_2,\\
 \mathcal U_D^{\rm long}
  &=[0,m-3]_2\cup\{m+2\}\cup[m+4,2m-2],\\
 \mathcal U_R^{\rm long}
  &=[1,m-4]_2\cup[m+4,2m-1]_2.
\end{align*}

For completeness, the following elementary identities are used only for
nonempty intervals.  If $u\le v$ and $u'\le v'$, then
\[
 [u,v]+[u',v']=[u+u',v+v'].
\]
If in addition $u\equiv v\pmod2$ and $u'\equiv v'\pmod2$, then
\[
 [u,v]_2+[u',v']_2=[u+u',v+v']_2.
\]
Finally, when $u\equiv v\pmod2$, $u\le v$, $u'\le v'$, and
$v'-u'\ge1$,
\[
        [u,v]_2+[u',v']=[u+u',v+v'].
\]
These identities, together with the decompositions of
$G_p,H_p,\mathcal D_a,\mathcal R_a$ in
\eqref{eq:j1-GH-table}--\eqref{eq:j1-DR-table}, yield the displayed
expansions.  For example, in the
strict interior case,
\[
 ([2,2a]_2+[2,m-2a])-m+1
   =[5-m,1]
\]
reduces to $[m+5,2m-1]\cup[0,1]$, and
\[
 \{m+2a\}+[2,m-2a]-m+1=[2a+3,m+1].
\]
The other terms are obtained in the same way, with their parity fixed by
the two step-$2$ progressions; the first and last terms are displayed
above.

Taking the unions and then adjoining the two boundary terms gives
\begin{align*}
 a=1:
 &\quad\Z_{2m}\setminus\mathcal F_p^{\rm short}
   =\{2,3,4,m+4\},\\
 &\quad\Z_{2m}\setminus\mathcal F_p^{\rm long}
   =\{1,2,3,m+2,m+3\};\\[1mm]
 2\le a,\ 2a+1<m:
 &\quad\Z_{2m}\setminus\mathcal F_p^{\rm short}
   =\{2a-2,2a,2a+1,2a+2\},\\
 &\quad\Z_{2m}\setminus\mathcal F_p^{\rm long}
   =\{2a-1,2a,2a+1,m+1\};\\[1mm]
 2a+1=m:
 &\quad\Z_{2m}\setminus\mathcal F_p^{\rm short}\\
 &\qquad=\{m-3,m-1,m\}\cup\{m+1,m+2,m+4\},\\
 &\quad\Z_{2m}\setminus\mathcal F_p^{\rm long}
   =\{m-2,m-1,m,m+1,m+3\}.
\end{align*}
These are exactly the identities in
\eqref{eq:j1-complement-identities}.  Every complement has size at most
six, so each forbidden set has at least $2m-6\ge m+2$ vertices.  Apply
Lemma~\ref{lem:all-split-exclusion}.
\end{proof}

It remains to remove the interval orbit.  Rotate it to
\[
        A=f=[0,m-2].
\]

\begin{lemma}[The interval orbit forces length $m$]
\label{lem:interval-m}
If $m\ge7$ and $f=[0,m-2]$, then length $m$ occurs.
\end{lemma}

\begin{proof}
Assume length $m$ is absent and put $s=m-1$, $p=1$.  Here
\[
 G_1=\{0,1\},
 \qquad H_1=[2,m-2],
\]
while the applicable row-first parameters are
$D=1,2,\ldots,m-3$ and the only applicable $f$-first parameter is $R=1$.
In \eqref{eq:all-split-short}, the row-first union is
$[m+2,2m-1]$, the $f$-first union is $[m-1,2m-5]$, and the two boundary
terms are $m+1,m+2$.  Therefore
\[
        \mathcal F_1^{\rm short}=[m-1,2m-1].
\]
Absence of the target forces $e_1\subseteq[0,m-2]=f$.  Equality of their
sizes gives $e_1=f$, contrary to simplicity and to $f\notin E(C)$.
\end{proof}

\begin{lemma}[The interval orbit forces length $m+2$]
\label{lem:interval-mplus2}
If $m\ge8$ and $f=[0,m-2]$, then length $m+2$ occurs.
\end{lemma}

\begin{proof}
Assume length $m+2$ is absent.  Fix $1\le p\le m-4$ and put $s=m-1$.
For the interval edge $A=[0,m-2]$ one has
\[
 G_p=[0,p],
 \qquad H_p=[p+1,m-2],
\]
with applicable parameters
\[
        D=1,\ldots,m-2-p,
        \qquad R=1,\ldots,p.
\]
In \eqref{eq:all-split-long}, the row-first translates cover
$[m+1,2m-2]$, the $f$-first translates cover $[m,2m-4]$, and both boundary
terms already lie in their union.  Hence
\[
        \mathcal F_p^{\rm long}=[m,2m-2],
        \qquad
        \Z_{2m}\setminus\mathcal F_p^{\rm long}
        =[0,m-1]\cup\{2m-1\}=:S.
\]
Thus $e_p\subseteq S$.

Because $p,p+1\in f$, replace $e_p$ by $f$ on the adjacent Hamiltonian pair
$\{p,p+1\}$.  This gives another Hamiltonian Berge cycle with $e_p$ unused.
If $e_p$ contained a pair at cyclic distance $m-1$, the chord lemma applied
to this new Hamiltonian cycle would give length $m+2$.  Therefore $e_p$ is
distance-$(m-1)$-free.

The only distance-$(m-1)$ pairs inside the $(m+1)$-set $S$ are
\[
        \{0,m-1\},
        \qquad
        \{m-2,2m-1\}.
\]
An $(m-1)$-subset of $S$ omits two vertices and must hit both pairs with its
omitted set.  There are four possibilities; one is $f$, and the other three
are
\[
\begin{aligned}
 E_1&=[0,m-3]\cup\{2m-1\},\\
 E_2&=[1,m-1],\\
 E_3&=[1,m-3]\cup\{m-1,2m-1\}.
\end{aligned}
\]
The $m-4$ labels $e_p$, $1\le p\le m-4$, are mutually distinct and are all
distinct from $f$.  For $m\ge8$ there are at least four such labels but only
three available sets, contradicting simplicity.
\end{proof}

\begin{corollary}[The two $j=1$ middle lengths]
\label{cor:j1-middle-lengths}
For every $m\ge8$, if the extra edge misses distance $m-1$, then both
lengths $m$ and $m+2$ occur.
\end{corollary}

\begin{proof}
Apply the normalizing dihedral automorphism to the entire labeled setup,
transporting Hamiltonian labels by $e'_{p+t}=\phi(e_p)$ under a rotation
$\phi(x)=x+t$ and by $e'_{c-p-1}=\phi(e_p)$ under a reflection
$\phi(x)=c-x$.  Noninterval states are covered by
Lemma~\ref{lem:j1-noninterval}; the interval state is covered by
Lemmas~\ref{lem:interval-m} and~\ref{lem:interval-mplus2}.
\end{proof}

\subsection{Completion of the near-central band}

\begin{theorem}[Six noncentral middle lengths]
\label{thm:six-middle-lengths}
In the even one-extra-edge setup, for every $m\ge50$ the six lengths
\[
        m-2,\quad m-1,\quad m,\quad m+2,\quad m+3,\quad m+4
\]
all occur.
\end{theorem}
\begin{proof}
For $j\in\{1,2,3\}$ put $s=m-j$.  If $s\in\Delta(f)$, the chord lemma gives
both associated lengths $s+1=m-j+1$ and
$2m-s+1=m+j+1$.  If $s\notin\Delta(f)$, apply
Corollary~\ref{cor:j1-middle-lengths} for $j=1$ and
Corollary~\ref{cor:j23-middle-lengths} for $j=2,3$.
\end{proof}

\section{Completion of the proof}\label{sec:assembly}

\begin{proof}[Proof of Theorem~\ref{thm:main-even}]
Take $m$ sufficiently large for Theorems~\ref{thm:outside-middle-band},~\ref{thm:central-middle-length}, and~\ref{thm:six-middle-lengths}.  The first of
these gives the endpoint lengths and every length outside the seven-element
middle band.  Theorem~\ref{thm:central-middle-length} gives $m+1$, and
Theorem~\ref{thm:six-middle-lengths} gives the other six middle lengths.
Thus $E(C)\cup\{f\}$ contains a Berge cycle of every length
$2,3,\ldots,2m$.
\end{proof}

\begin{proof}[Proof of Theorem~\ref{thm:main}]
For odd $n$, Corollary~\ref{cor:odd-one-extra} applies.  For even $n$, write
$n=2m$ and apply Theorem~\ref{thm:main-even}.  Taking $n_0$ large enough for
the even theorem proves the result in both parities.
\end{proof}

\section{Concluding remarks}

The contrast between the two parity classes is structural.  In odd order,
a half-size set misses at most one cyclic distance, and an alternating
matching exchange can prescribe the distance realized by an unused edge.
In even order, an $(m-1)$-set may miss every odd distance.  The two-gap
construction replaces the one-chord exchange, while the central
obstruction is converted into a reflected additive model whose bridge
supply is controlled by sum-free stability.

Theorem~\ref{thm:main} establishes the asymptotic statement asked by
Bailey, Hollars, Li and Luo.  Determining the smallest even order for which
the conclusion always holds remains a separate finite problem.

\end{document}